\documentclass[intlimits,twoside]{amsart}

\usepackage[all]{paper_diening}

\usepackage{color}

\providecommand{\Rnn}{\setR^{n \times n}}
\providecommand{\Rnnsym}{\setR^{n \times n}_{\sym}}

\usepackage[english]{babel}
\providecommand{\logmean}[1]{\mean{#1}^{\log}}

\providecommand{\bigglogmean}[1]{\biggmean{#1}^{\log}}

\newcommand{\hatxx}{\hat{x} \otimes \hat{x}}

\newcommand{\idhatxx}{\identity - \hat{x} \otimes \hat{x}}

\numberwithin{equation}{section}

\allowdisplaybreaks

\begin{document}

\title{Elliptic Equations With  Degenerate  weights}

\author[Kh. Balci, Diening, Giova, Passarelli di Napoli]{Anna Kh.~Balci \and Lars Diening  \and Raffaella Giova \and Antonia  Passarelli di Napoli}

% \normalsize{Dipartimento di Matematica e Applicazioni "R.
% Caccioppoli"} \\ \normalsize{Universit\`{a} di Napoli ``Federico
% II", via Cintia - 80126 Napoli}\\ \\ \normalsize{Faculty for Mathematics - University of Bielefeld}\\
% \normalsize{Postfach 10 01 31, D-33501 Bielefeld
% }
%  \\ \normalsize{e-mail:
% \email{akhripun@math.uni-bielefeld.de}
% \email{lars.diening@uni-bielefeld.de}
% \email{raffaella.giova@uniparthenope.it}
% \email{antpassa@unina.it}

\begin{abstract}
  We obtain new local \Calderon{}-Zygmund estimates for elliptic
  equations with matrix-valued weights for linear as well as
  non-linear equations. We introduce a novel $\log$-$\setBMO$
  condition on the weight~$\mathbb M$. In particular, we assume
  smallness of the logarithm of the matrix-valued weight
  in~$\setBMO$. This allows to include degenerate, discontinuous
  weights. We provide examples that show the sharpness of the
  estimates in terms of the $\log$-$\setBMO$-norm.
\end{abstract}



  \thanks{The research of A. Kh.~Balci and Lars Diening was supported
    by the DFG through the CRC 1283.  R. Giova and A. Passarelli di
    Napoli have been partially supported by the Gruppo Nazionale per
    l’Analisi Matematica, la Probabilit`a e le loro Applicazioni
    (GNAMPA) of the Istituto Nazionale di Alta Matematica (INdAM).
    R. Giova has been partially supported by Universit`a degli Studi
    di Napoli Parthenope through the projects “Sostegno alla Ricerca
    individuale”(2015 - 2016 - 2017) and “Sostenibilit`a, esternalit`a
    e uso efficiente delle risorse ambientali”(2017-2019)}
  
\maketitle

%\tableofcontents

\section{Introduction and Statement of the Results}
\label{sec:intr-stat-results}

We study weak solutions of elliptic equations with degenerate
matrix-valued weights. We consider both the linear case as well as the
non-linear case. The main concern is the transfer of regularity from
the data $G \,:\, \Omega \to \setR^n$ to the weak solution
$u \,:\, \Omega \to \setR$ of the equation
\begin{align}
  \label{eq:sysLinear}
  -\divergence (\bbA(x) \nabla u)
  &=
  -\divergence ( \bbA(x) G),
\end{align}
in the linear case  and of
the equation
\begin{align}
  \label{eq:sysM}
  -\divergence \big( \abs{\bbM(x) \nabla u}^{p-2} \bbM^2(x) \nabla u\big)
  &=
  -\divergence \big( \abs{\bbM(x) G}^{p-2} \bbM^2(x) G\big),
\end{align}
for the non-linear case.  Here $\Omega \subset \Rn$ is a domain,
$1<p<\infty$, $\bbA, \bbM\,:\, \Omega \to \Rnnsym$ are symmetric,
positive definite, matrix-valued weights (almost everywhere) with
$\bbM = \bbA^{\frac 12}$, and~$\nabla u$ is the column vector
$(\partial_1 u, \dots, \partial_n u)^T$ and $\abs{\cdot}$ denote the
usual euclidean distance on~$\Rn$.
We do not need to specify boundary values, since we only look at local
solutions.  If $p=2$, then the non-linear equation~\eqref{eq:sysM} reduces
to linear one~\eqref{eq:sysLinear}.

The main objective of our paper is to transfer regularity of the
data~$G$ in terms of weighted Lebesgue spaces~$L^\rho_\omega$ to the
gradient of the solution. Very roughly speaking we want to prove an
estimate of the form
\begin{align}
  \label{eq:Calderon}
  \norm{\nabla u}_{L_\omega^\rho(B)}\lesssim \norm{G}_{L_\omega^\rho(2B)} +
  \text{lower order terms}
\end{align}
where~$L_\omega^\rho$ is the natural corresponding weighted Lebesgue
space. We treat the weights~$\omega$ in the multiplicative sense,
i.e.
$\norm{\nabla u}_{L_\omega^\rho(B)} = \norm{ \omega \nabla
  u}_{L^\rho}(B)$ and $L^\rho_\omega(B)$ corresponds to
$L^\rho(B,\omega^\rho \,dx)$.

Let us start with the linear case in the form of~\eqref{eq:sysLinear}
with~$\bbA=\bbM=\identity$ (unweighted case; $\identity$ is the
identity matrix), then~\eqref{eq:Calderon} just follows from the
standard~$L^q$-estimates for the Laplacian using \Calderon{}-Zygmund
theory for all $\rho \in (1,\infty)$.  If $\bbA$ is a uniformly
elliptic weight~$\bbA(x)$, i.e.
$\lambda_1 \abs{\xi}^2 \leq \skp{\bbA(x) \xi}{\xi} \leq \Lambda_1
\abs{\xi}^2$ for all~$x \in \Omega$ and all~$\xi \in \Rn$, then Meyers
in \cite{Mey63} proved~\eqref{eq:Calderon} for
$\rho \in [2,2+\epsilon)$ for some~$\epsilon>0$. He achieved this by
representing the equation with weight as a perturbation of the
Laplacian. The same idea earlier was used by Boyarski\u{\i} in
\cite{Boy57} for the problem on the plane. The case of bounded and
uniformly continuous (and therefore non-degenerate) positive definite
weights has been studied for example by Morrey~\cite{Mor66}. This has
been extended by Di~Fazio~\cite{DiF96} to the case of uniformly
elliptic weights, i.e.
$\lambda_1 \abs{\xi}^2 \leq \skp{\bbA(x) \xi}{\xi} \leq \Lambda_1
\abs{\xi}^2$ for all~$x \in \Omega$ and all~$\xi \in \Rn$, with
$\bbA \in \setVMO$ (vanishing mean oscillation). The case of systems
has been obtained by Di~Fazio, Fanciullo and Zamboni
\cite{DiFFanZam13}.  The global result for equations has been obtained
by Iwaniec and Sbordone~\cite{IwaSbo98}.

Due to the assumed uniform ellipticity, the results above exclude the
possibility of degenerate weights like
$\abs{x}^{\pm \epsilon} \identity$.  Fabes, Kenig and
Serapioni~\cite{FabKenSer82} studied the case, where
\begin{align}
  \label{eq:ass-fabes-kenig-serapioni}
  \Lambda^{-1} \mu(x) \abs{\xi}^2 \leq \skp{\bbA(x) \xi}{\xi} \leq
  \Lambda \mu(x) \abs{\xi}^2
\end{align}
for some non-negative weight~$\mu$. This is equivalent to say that
$\bbA(x)$ has a uniformly bounded condition number~$\Lambda^2$.
Fabes, Kenig and Serapioni proved that~$u$ is H\"older continuous
provided that~$\mu$ is in the Muckenhoupt class~$\mathcal{A}_2$. Cao,
Mengesha and Phan~\cite{CaoMenPha18} have considered gradient estimates
in the linear case also under the
condition~\eqref{eq:ass-fabes-kenig-serapioni}. Additionally, they
assumed that $\mu$ is of Muckenhoupt class~$\mathcal{A}_2$ and that
$\bbA$ has small~$\setBMO^2_\mu$ norm, where
\begin{align*}
  \abs{\bbA}_{\setBMO^s_\mu} &= \sup_B
                                \bigg( \frac{1}{\mu(B)} \int_B
                                \biggabs{\frac{\abs{\bbA(x) -
                                \mean{\bbA}_B}}{\mu(x)}}^s \mu(x)\,dx
                                \bigg)^{\frac 1s}
\end{align*}
for $s \geq 1$, where the supremum is taken over all balls~$B$ and
$\mean{\bbA}_B$ denotes the mean-value over the ball~$B$. (If $\mu$ is
of Muckenhoupt class~$\mathcal{A}_1$, then
$\setBMO^1_\mu = \setBMO^s_\mu$ with equivalence of norms,
see~\cite{GarJos89}.) Under these conditions Cao, Mengesha and Phan
proved that $\abs{G}^q \mu \in L^1_{\loc}$ implies
$\abs{\nabla u}^q \mu \in L^1_{\loc}$. Their condition allowed to include weights like~$\abs{x}^{\pm \epsilon}\identity$ for small~$\epsilon>0$.  The case of systems has been
covered by the same authors in~\cite{CaoMenPha19}. Our 
condition on the weight~$\bbA$ differs somehow from the previous
ones. Instead of a~$\setBMO$ or $\setBMO^2_\mu$ smallness condition
for~$\bbA$, we use a~$\setBMO$ smallness condition on its logarithm
$\log \bbA$ or equivalently on $\log \bbM = \frac 12 \log \bbA$. This
new \emph{$\log$-$\setBMO$-condition} allows us also to include the
degenerate weights, for example
$\bbM(x) := \abs{x}^{\epsilon} \identity$ and
$\bbM(x) := \abs{x}^{-\epsilon} \identity$ for small~$\epsilon>0$. We
will show in Section~\ref{sec:counterexample} by a counterexample that
this $\log$-$\setBMO$ condition is sharp in terms of the achievable
higher integrability exponent~$q$. 

To our knowledge the
$\log$-$\setBMO$ condition is new even in the context of linear
equations.

Let us also mention that our $\log$-$\setBMO$ condition as well as
the $\setBMO^2_\mu$-condition of Cao, Mengesha and Phan are invariant
under scaling of the equation in the following sense. If we
scale~$\bbM$ by a factor~$t>0$ and $u$ and $G$ by~$1/t$ (which will
scale~$\omega$ by~$t$ and $\bbA$ and~$\mu$ by $\sqrt{t}$), then the
equation remains valid. Moreover, $\abs{\nabla u} \omega$ and
$\abs{G} \omega$ are scaling invariant. Thus, the condition on the
weight~$\bbM$ for the higher integrability of~$\abs{\nabla u} \omega$
with respect to~$\abs{G} \omega$ should be invariant under this
scaling. Now, our $\log$-$\setBMO$ condition is scaling invariant,
since $\abs{\log(t\bbM)}_{\setBMO}=\abs{\log(\bbM)}_{\setBMO}$.  Note
however, that the condition~$\bbA \in \setVMO$ by Di~Fazio\cite{DiF96}
is not scaling invariant and therefore not natural.

Our main result differs also from the one of Cao, Mengesha and
Phan~\cite{CaoMenPha18} since we treat the weight~$\omega$ rather as a
multiplier than a measure. Cao, Mengesha and Phan show for~$p=2$ that
$\abs{G}^{\rho} \omega^2 \in L^1_{\loc}$ implies
$\abs{\nabla u}^{\rho} \omega^2 \in L^1_{\loc}$ (recall
$\mu = \Lambda^{-1} \omega^2$), so the weight stays the same for all
exponents. We, on the other hand, show that
$\abs{G}\, \omega \in L^\rho_{\loc}$ implies
$\abs{\nabla u}\, \omega \in L^\rho_{\loc}$. So in our situation the
original weight~$\omega^2$ (for $p=2$) changes to~$\omega^{\rho}$.

Let us also mention that Baison, Clop, Giova, Orobitg and Passarelli
di Napoli in \cite{BCGOP17} derived estimates in Besov and Triebel-
Lizorkin spaces for slightly non-linear system (non-linear but linear
growth) for uniformly elliptic weights given in a suitable Besov space.

We now turn to the non-linear equations in the form
of~\eqref{eq:sysM}. Recall that the linear
equation~\eqref{eq:sysLinear} is just the special case~$p=2$. Let us
abbreviate
\begin{align}
  \label{eq:def-calA} 
  \begin{aligned}
    A(\xi) &:= \abs{\xi}^{p-2} \xi
    \\
    \mathcal{A}(x,\xi) &:= |\bbM(x)\xi|^{p-2} \bbM^2 \xi = \bbM
    A(\bbM \xi).
  \end{aligned}
\end{align}
Note that we use the upright letter~$A$ for the unweighted version and
the calligraphic letter~$\mathcal{A}$ for the weighted version.
Then we can rewrite~\eqref{eq:sysM} as
\begin{align}
  \label{eq:calA}
  -\divergence \mathcal{A}(\cdot,\nabla u) &=
  -\divergence \mathcal{A}(\cdot,G).
\end{align}
Sometimes in literature, e.g.~\cite{KinZho99}, the system is also
given as
\begin{align}
  \label{eq:sysA}
  -\divergence \big(\skp{\bbA \nabla u}{\nabla u}^{\frac {p-2}{2}}
  \bbA \nabla u \big)
  &= 
  -\divergence \big(\abs{F}^{p-2} F\big)
\end{align}
with some positive definite matrix-valued weight
$\bbA\,:\, \Omega \to \Rnnsym$ (almost everywhere) and
$F\,:\, \Omega \to \Rn$ is the given data. With the transformation
$\bbM = \bbA^{\frac 12}$ and $\mathcal{A}(\cdot,G) = A(F)$ we
can pass from~\eqref{eq:sysA} to~\eqref{eq:sysM} and vice versa.

Note that~$u$ is just the local minimizer of the energy
\begin{align*}
  \mathcal{J}(v) &:= \int_\Omega \tfrac 1p \big( \skp{\bbA \nabla
                   v}{\nabla v}^{\frac p2} \,dx -
                   \int_\Omega \abs{F}^{p-2}F \cdot \nabla v\,dx
  \\
                 &= \int_\Omega \tfrac 1p \abs{\bbM \nabla v}^p\,dx -
                   \int_\Omega \abs{\bbM G}^{p-2} \bbM G \cdot (\bbM
                   \nabla v)\,dx. 
\end{align*}

We assume that our matrix-valued weight has a uniformly bounded
condition number, i.e. 
\begin{alignat}{2}
  \label{eq:ass-M}
  \abs{\bbM(x)} \abs{\bbM^{-1}(x)} &\leq \Lambda &&\text{for
                                     all $x \in \Omega$,}
  \\
  \intertext{or equivalently}
  \label{eq:ass-A}
  \abs{\bbA(x)} \abs{\bbA^{-1}(x)} &\leq \Lambda^2 &\qquad &\text{for
                                     all $x \in \Omega$.}
\end{alignat}
Let us define
\begin{align}
  \label{eq:def-omega}
  \omega(x) &:= \abs{\bbM(x)} = \abs{\bbA(x)}^{\frac 12}.
\end{align}
Then~\eqref{eq:ass-A} is also equivalent to
\begin{align}
  \label{eq:ass-A2}
  \Lambda^{-2} \omega^2(x) \abs{\xi}^2 \leq \skp{\bbA(x) \xi}{\xi} \leq
  \omega^2(x) \abs{\xi}^2 \qquad \text{for all $\xi \in \Rn$ and almost all~$x
  \in \Omega$.}
\end{align}
Note that~\eqref{eq:ass-fabes-kenig-serapioni} with the choice
$\mu(x) := \Lambda^{-1} \omega^2(x)$ is exactly our
condition~\eqref{eq:ass-A2}.  Since $\bbM(x)$ is positive definite, we
can define its logarithm~$\log \bbM(x)$ either by transformation to a
diagonal matrix or by Taylor series. Note that
$\log \bbM = \frac 12 \log \bbA$.
As usual let us denote  by $\dashint_B \dotsm
\,dx$ the mean value integral.

Our main result is as follows.
\begin{theorem}[Linear Case]
  \label{thm:main-linear}
  Let $u$ be a local weak solution of~\eqref{eq:sysLinear}, let~$\bbA$
  satisfy~\eqref{eq:ass-A}, define~$\omega$ by~\eqref{eq:def-omega}.
  Then there exists $\kappa = \kappa(p,n, \Lambda)$ such that for all
  balls~$B_0$ with $4B_0 \subset \Omega$ and all $\rho \in (1,\infty)$
  with
  \begin{align}
    \label{eq:smallness}
    \abs{\log \bbA}_{\setBMO(4B_0)} \leq \kappa \min \biggset{
    \frac{1}{\rho}, 1 - \frac{1}{\rho}}
  \end{align}
  there holds
  \begin{align}
    \label{eq:main-linear}
    \bigg(\dashint_{B_0} \big( \abs{\nabla u}\, \omega
    \big)^\rho\,dx\bigg)^{\frac 1\rho} &\leq c_{\rho_0} \dashint_{2 B_0} 
                                   \abs{\nabla u}\, \omega
                                         \,dx + c_{q_0} \bigg( \dashint_{2 B_0}
                                   \big( \abs{G}\, \omega\big)^\rho\,dx
                                   \bigg)^{\frac 1 \rho}. 
  \end{align}
  for all balls~$B_0$ with $4B_0 \subset \Omega$.
\end{theorem}
\begin{theorem}[Non-Linear Case]
  \label{thm:main}
  Let $u$ be a local weak solution of~\eqref{eq:sysM}, let~$\bbM$
  satisfy~\eqref{eq:ass-M}, define~$\omega$ by~\eqref{eq:def-omega}.
  Then there exists $\kappa = \kappa(p,n, \Lambda)$ such that for
  all balls~$B_0$ with $8B_0 \subset \Omega$ and all
  $\rho \in [p,\infty)$ with
  \begin{align}
    \label{eq:smallness-2}
    \abs{\log \bbM}_{\setBMO(8B_0)} \leq \kappa \frac{1}{\rho}
  \end{align}
  there holds
  \begin{align}
    \label{eq:main}
    \bigg(\dashint_{\frac 12 B_0} \big( \abs{\nabla u}\, \omega
    \big)^\rho\,dx\bigg)^{\frac 1 \rho}
    &\leq c_\rho \dashint_{4 B_0} 
      \abs{\nabla u}\, \omega \,dx
      + c_\rho \bigg( \dashint_{4
      B_0} 
      \big( \abs{G}\, \omega\big)^\rho\,dx \bigg)^{\frac 1 \rho}
  \end{align}
  for all balls~$B_0$ with $8B_0 \subset \Omega$, where $c_\rho =
  c_\rho(p,n,\Lambda,\rho)$. The constant $c_\rho$ is continuous in~$\rho$.
\end{theorem}
Note that Theorem~\ref{thm:main-linear} holds for all $\rho\in
(1,\infty)$, while the non-linear case of Theorem~\ref{thm:main}
requires~$\rho \geq p$. We write below more on this difference.

The case $1< p< \infty$ with $\bbA=\bbM=\identity$ (unweighted case)
has been obtained by Iwaniec~\cite{Iwa83} and Di Benedetto and
Manfredi~\cite{DiBMan93}.  The limiting case~$\rho=\infty$ is slightly
different and better expressed in terms
of~$\mathcal{A}(\cdot,\nabla u)$ and~$F$. It has been shown by Di
Benedetto and Manfredi~\cite{DiBMan93} for~$p>2$ and by Diening,
\Kaplicky{} and Schwarzacher~\cite{DieKapSch12} for all $1<p<\infty$
that $F \in \setBMO$ implies
$\mathcal{A}(\cdot,\nabla u) \in \setBMO$. In~\cite{DieKapSch12} is
has also been shown that~$\setBMO$ can be replaced by $C^{0,\alpha}$
for small~$\alpha>0$. All of these results are also consequences of
the point-wise estimates obtained in~\cite{BreCiaDieKuuSch18} by
Breit, Cianchi, Diening, Kuusi and Schwarzacher. The same authors
proved estimates up to the boundary
in~\cite{BreCiaDieKuuSch18b}. \Calderon{}-Zygmund estimates in the
space~$W^{1,2}$ have been studied by Cianchi and
Maz'ya~\cite{CiaMaz19}. Estimates in Besov and Triebel-Lizorkin spaces
up to differentiability with arbitrary integrability one have been
studied in the planar case of equations for $p>2$ by Balci, Diening
and Weimar~\cite{BalDieWei20}.  Gradient estimates for the right hand
side in non-divergence form were obtained for equations by Kuusi and
Mingione see \cite{KuuMin13}, \cite{KuuMin12}. The case of systems was
considered by Duzaar and Mingione in \cite{DuzMin10} and by Kuusi and
Mingione in \cite{KuuMin14}, \cite{KuuMin18}.

Let us turn to the weighted case.  In \cite{KinZho99} Kinnunen and
Zhou extended~\eqref{eq:main} to the case~$1<p<\infty$, also for
uniformly elliptic weights with $\bbA \in \setVMO$ (vanishing mean
oscillation). It is also enough to assume that~$\bbA$ has small
$\setBMO$-norm. They have also obtained global results
in~\cite{KinZho01}. Note that both conditions are not scaling
invariant (as mentioned above).

The condition~$\rho \geq p$ in our theorem is due to the non-linear
situation~$p \neq 2$. The case~$\rho=p$ corresponds to the context of
weak solutions, while~$\max \set{1,p-1}<\rho<p$ corresponds the case
of very weak solutions.  Although it is conjectured by Iwaniec and
Sbordone~\cite{IwaSbo94} that $\rho > \max \set{1,p-1}$ should be the
maximal range for~$\rho$, this has not been shown yet. In the same
paper they prove~\eqref{eq:main} in the unweighted case
for~$\rho> p-\epsilon$ for small~$\epsilon>0$. A qualitative control
of~$\epsilon$ has been obtained in~\cite{KinZho97} by Kinnunen and
Zhou, which implies the optimal range~$\rho > \max \set{1,p-1}$ but
only if~$\abs{p-2}$ is small. This results has been extended to
uniformly elliptic weights with $\bbA \in \setVMO$ by Greco and
Verde~\cite{GreVer00}.  

There are only a few results on the non-linear case with degenerate
weights. Cruz-Uribe, Moen and Naibo proved H\"older continuity of the
solution for $1<p< \infty$ in \cite{CruMoeNai13} also using a
Muckenhoupt condition.  For matrix-valued weights there exists also a
weaker notion of matrix-valued Muckenhoupt classes~$\mathcal{A}_p$ by
Roudenko~\cite{Rou03}.  This weaker notion was for example used by
Cruz-Uribe, Moen and Rodney~\cite{CruMoeRod16} to prove partial
regularity for mappings of finite distortion, where~\eqref{eq:ass-A2}
is replaced by a condition with different lower and upper growth.

The outline of this article is as follows.  In
 Section~\ref{sec:weights} we introduce and present
new facts  on scalar and matrix-valued weights and their logarithm. This also includes
\Poincare{}-type estimates and new John-Nirenberg type estimates.

In  Section~\ref{sec:cald-zygm-estim} we then derive our
\Calderon{}-Zygmund estimates. We begin in 
Subsection~\ref{ssec:caccioppoli-estimate} with Caccioppoli and reverse
H\"older inequalities.  The comparison system is constructed in 
Subsection~\ref{sec:comparison-estimate}. The comparison estimate is
proved in Proposition~\ref{pro:comparison}, Subsection~\ref{sec:comparison-estimate} and conclude
the decay estimate in Subsection~\ref{ssec:decay-estimates}. Finally,
the proof of our main theorems are presented in the
Subsection~\ref{ssec:main}.

In the final  Section~\ref{sec:counterexample} we show by means of
examples that our results are sharp. In particular, we show that the
smallness condition on~$\rho \abs{\log \bbA}_{\setBMO}$ is optimal to
obtain $L^\rho$ integrability of $\abs{\nabla u}^p \omega^p$.

%% ------------------------------------------------------------
\section{On Scalar and Matrix-Valued Weights}
\label{sec:weights}
%% ------------------------------------------------------------

In this section we present the necessary tools on scalar and
matrix-valued weights. We also introduce a novel smallness condition
in terms of the logarithm of the weight. After this we show that this
condition implies suitable \Poincare{} type estimates.

%% ------------------------------------------------------------
\subsection{Matrix-Valued Weights and Logarithms}
\label{sec:notat-basic-prop}

%% ------------------------------------------------------------

By $\Rnnsym$ we denote the symmetric, real-valued matrices.  By
$\setR^{n \times n}_{\geq 0}$ we denote the cone of symmetric,
real-valued, positive semidefinite matrices and by
$\setR^{n\times n}_{> 0}$ the subset of positive definite matrices.
For $\bbX, \bbY \in \Rnnsym$, we write $\bbX \geq \bbY$ if $\bbX -\bbY
\in \setR^{n \times n}_{\geq 0}$. 

We say that $\bbM \,:\, \Omega \to \Rnn_{\geq 0}$ is a
\emph{(matrix-valued) weight} if $\bbM$ is almost everywhere positive
definite. We say that~$\omega\,:\, \Omega \to [0,\infty)$ is a
\emph{(scalar) weight} if~$\omega$ is positive almost everywhere.

For simplicity we assume in this section that our weights are defined
on the whole~$\Rn$ (instead of the subset~$\Omega$). If they are defined
only on~$\Omega$, they have to be extended in a suitable way
to~$\Rn$. This is not difficult due to the locality of our main
theorem, Theorem~\ref{thm:main}. 

By~$\abs{\cdot}$ we denote the euclidean norm on~$\setR^n$. For
$\bbL \in \Rnnsym$, let $\abs{\bbL}$ denote the spectral norm (which is
just the matrix norm induced by the euclidean norm for vectors).  We
write $B_R(x_0)\subset\Rn$ for the open ball of radius $R>0$ and
center $x_0\in\Rn$. For a ball~$B$ we denote by~$r_B$ the radius and
by~$x_B$ the center of~$B$. For the mean value of a function over a
ball~$B$ we write $\mean{f}_B :=\dashint_{B}f(x)\,dx$. We write
$\indicator_U$ for the indicator function of the set~$U$.

We will denote by $c$ a general constant that may vary on different
occasions, even within the same line of estimates.  Relevant
dependencies on parameters and special constants will be suitably
emphasized using parentheses or subscripts. We also write $f \lesssim
g$ if $f \leq c\,g$. We write $f \eqsim g$ if $f \lesssim g$ and $g
\lesssim f$.

By $L^p(\Rn)$ we denote the usual Lebesgue space with
norm~$\norm{\cdot}_p$ and by $L^p_{\loc}(\Rn)$ its local version
($L^p$ on compact subsets). By $p'$ we denote the conjugate exponent.

For $1< p < \infty$ and a weight $\omega \in L^p_{\loc}(\Rn)$ with
$\omega^{-1} \in L^{p'}_{\loc}(\Rn)$ we define the weighted spaces
\begin{align*}
  L^p_\omega(\Rn) &:= \set{f\,:\, \Rn \to \setR\,:\, \omega f \in L^p(\Rn)}
\end{align*}
with norm $\norm{f}_{p,\omega} := \norm{f\, \omega}_p$. We write
$L^p(\Rn, \mu)$ for the $L^p$-space with measure~$\mu$. So
$L^p_\omega(\Rn) = L^p(\Rn, \omega^p \,dx)$.  Note that we
use~$\omega$ as a multiplicative weight (not as a measure).  The dual
space of $L^p_\omega(\Rn)$ is $L^{p'}_{1/\omega}(\Rn)$. Both
$L^p_\omega(\Rn)$ and $L^{p'}_{1/\omega}(\Rn)$ are Banach functions
spaces mapping to~$L^1_{\loc}(\Rn)$. Let $W^{1,p}(\Omega)$ denote the
usual Sobolev space. Let~$W^{1,p}_{\loc}(\Omega)$ be its local version
and $W^{1,p}_0(\Omega)$ be the one with zero boundary values.
Let~$W^{1,p}_\omega(\Omega)$ be the weighted Sobolev space, which
consists of functions~$u \in W^{1,1}(\Omega)$ such that
$u, \abs{\nabla u} \in L^p_\omega(\Omega)$. We
equip~$W^{1,p}_\omega(\Omega)$ with the norm
$\norm{u}_{L^p_\omega(\Omega)} + \norm{\nabla
  u}_{L^p_\omega(\Omega)}$. Let $W^{1,p}_{0,\omega}(\Omega)$ denote
the subspace of functions with zero boundary values.

For every~$\bbL \in \Rnnsym$ we can consider the matrix
exponential~$\exp(\bbL) \in \Rnn_{>0}$, i.e
$\exp \,:\, \Rnnsym \to \Rnn_{>0}$. Moreover, there exists a unique
inverse mapping $\log\,:\, \Rnn_{>0} \to \Rnnsym$.  Thus, since
$\bbM\,:\, \Rn \to \Rnnsym$ is almost everywhere positive definite, we
can define its logarithm $\log \bbM\,:\, \Rn \to \Rnnsym$.  Both
$\exp$ and $\log$ can be defined by transformation to a diagonal
matrix or by Taylor series.

Of particular interest to us are the logarithm means of~$\omega$ and
$\bbM$. We define the logarithmic means
\begin{align}
  \label{eq:logmean}
  \begin{aligned}
    \logmean{\omega}_B &:= \exp( \mean{\log \omega}_B),
    \\
    \logmean{\bbM}_B &:= \exp( \mean{\log \bbM}_B).
  \end{aligned}
\end{align}
Recall, that the dual space of $L^p_\omega$ is $L^{p'}_{1/\omega}$.
It is interesting to observe, that the logarithmic mean is compatible
with this operation, since
\begin{align}
  \label{eq:mean-omega-inv}
  \bigglogmean{\frac{1}{\omega}}_B &= \exp( -\mean{\log \omega}_B) =
                                 \frac{1}{\logmean{\omega}_B}. 
\end{align}
The logarithmic mean also commutes with inversion.  Indeed, using the identities
 $\log(\bbM^{-1}) = -\log \bbM$ and
$(\exp(\bbL))^{-1} = \exp(-\bbL)$ we get
\begin{align*}
  \logmean{\bbM^{-1}}_B
  &=  \exp \big( -\mean{\log (\bbM)}_B\big)
  = \big(\exp (\mean{\log (\bbM)}_B)\big)^{-1} = (\logmean{\bbM}_{B})^{-1}.
\end{align*}

\subsection{Muckenhoupt Weights}
\label{sec:muckenhoupt-weights}

We give a brief review on Muckenhoupt weights. Let $1<p<\infty$. A
weight~$\mu\in L^1_{\loc}(\Rn)$ is called
an~$\mathcal{A}_p$-Muckenhoupt weight if and only if
\begin{align*}
  [\mu]_{\mathcal{A}_p} &:= \sup_B\bigg(\dashint_B
                \mu\,dx\bigg)\bigg(\dashint_B \mu^{-\frac{1}{p-1}}\,dx\bigg)^{p-1}
                < \infty
\end{align*}
where the supremum is taken over all balls~$B$.

If $\mu$ is an~$\mathcal{A}_p$-Muckenhoupt weight then the maximal
operator~$M$ is bounded on~$L^p(\Rn,\mu)$.  Let us reformulate it in the
language of $L^p_\omega(\Rn)$. The weight $\omega^p$ is
an~$\mathcal{A}_p$-Muckenhoupt weight if and only if
\begin{align}
  \label{eq:def-muckenhoupt}
  [\omega^p]_{\mathcal{A}_p}^{\frac 1p} &= \sup_B\bigg(\dashint_B
                \omega^p\,dx\bigg)^{\frac 1p} \bigg(\dashint_B
                     \omega^{-p'}\,dx\bigg)^{\frac{1}{p'}} 
                < \infty
\end{align}
The property of being a Muckenhoupt weight can also be characterized
by its logarithmic means. Indeed, if~$\omega^p$ is
an~$\mathcal{A}_p$-Muckenhoupt weight, then by the help of Jensen's
inequality for all balls~$B$
\begin{align}
  \label{eq:muckenhoupt-log}
  \begin{aligned}
    \bigg( \dashint_B \omega^p \,dx \bigg)^{\frac 1p} &\leq c_1
    \logmean{\omega}_B,
    \\
    \bigg( \dashint_B \omega^{-p'} \,dx \bigg)^{\frac 1{p'}} &\leq c_2
    \logmean{\omega^{-1}}_B = c_2 \frac{1}{\logmean{\omega}_B}.
  \end{aligned}
\end{align}
with $c_1,c_2 = [\omega^p]_{\mathcal{A}_p}^{\frac 1p}$. On the other hand,
if~\eqref{eq:muckenhoupt-log} holds, then $\omega^p$ is
an $\mathcal{A}_p$-Muckenhoupt weight and $[\omega^p]_{\mathcal{A}_p}^{\frac
  1p} \leq c_1 c_2$ using $\logmean{\omega}_B \logmean{\omega^{-1}}_B=1$.

\subsection{Weighted \Poincare{} Estimate}
\label{sec:poinc-type-estim}

In this section we present a \Poincare{} type estimate in terms of
multiplicative weights.  The following Proposition is a scaling invariant
version of \cite[Theorem~3.3]{DreDur08}.

\begin{proposition}
  \label{pro:DreDur}
  Let $1 < p < \infty$ and $\theta \in (0,1]$ such that
  $\theta p \geq \max \set{1,\frac{np}{n+p}}$. Furthermore, let $B$ be
  a ball and assume that~$\omega$ is a weight on~$2B$ with
  \begin{align}
    \label{eq:DreDur}
    \sup_{B' \subset 2B} \bigg(\dashint_{B'}  \omega^p\,dx \bigg)^{\frac 1p}
    \bigg(\dashint_{B'}  \omega^{-(\theta p)'}\,dx \bigg)^{\frac 1{(\theta p)'}} &\leq c_1,
  \end{align}
  where the supremum is taken over all balls~$B'$ contained in~$2B$.
  Then
  \begin{align*}
    \bigg(\dashint_{B} \biggabs{ \frac{u-\mean{u}_{B}}{r_B}}^p
    \omega^p  \,dx\bigg)^{\frac 1p}
    &\le
      c_2\, \bigg(\dashint_{B} \big(\abs{\nabla
      u}\, \omega \big)^{\theta p} \,dx\bigg)^{\frac{1}{\theta p}},
  \end{align*}
  where $c_2=c_2(c_1,n,p)$.
\end{proposition}
\begin{proof}
  The result follows from~\cite[Theorem~3.3]{DreDur08} for a fixed
  ball with radius~$1$, which is formulated in a slightly different
  way. Since the condition~\eqref{eq:DreDur} is scaling invariant
  w.r.t. $x \leftrightarrow Rx$, we obtain the general estimate simply
  by scaling. Note that in the statement of
  ~\cite[Theorem~3.3]{DreDur08} the case~$\alpha=0$ (in notation  from~\cite{DreDur08}), which we need is
  excluded in the statement but included in the proof. For the sake of
  completeness let us restate their proof in a scaling invariant
  formulation.

  Recall that the Riesz potential of a measurable function~$f\in \setR^n$ is
  \begin{align*}
    (I_1 f)(x):=\int_{\setR^n} \frac{f(y)}{
    \abs{x-y}^{n-1}}\, dy.
  \end{align*}
  We use the following well-known estimate (see for example~\cite[Section~15.23]{HeiKilMar06}) 
  \begin{align}
    \label{eq:Riesz}
    \dashint_B \abs{v(x)-v(y)}\, dy\le c\,\int_B \frac{\abs{\nabla
    v(y)}}{\abs{x-y}^{n-1}}\, dy
    = c\, I_1\big( \indicator_B \abs{\nabla v}\big)(x).  
  \end{align}
  Let $g \in L^{p'}_{\frac 1 \omega}(B)$ with
  \begin{align*}
    \bigg(\dashint_B \Bigabs{\frac{g(x)}{\omega(x)}}^{p'}\,dx
    \bigg)^{\frac 1{p'}} &\leq 1.
  \end{align*}
  Applying \eqref{eq:Riesz}, we get
  \begin{align*}
    \biggabs{ \dashint_B (v(x) - \mean{v}_{B})g(x)\,dx}
    &\lesssim
      \biggabs{ \dashint_B I_1(\indicator_B \nabla v)(x) g(x)\,dx}
    \\
    &=  \biggabs{ \dashint_B (\nabla v)(x) I_1 (\indicator_B
      g)(x)\,dx }
    \\
    &\leq c\, \bigg(\dashint_B (\abs{\nabla v}\, \omega)^{\theta p} \,dx\bigg)^{\frac
      1{\theta p}}
      \bigg(\dashint_B \biggabs{\frac{I_1 (\indicator_B
      g)(x)}{\omega(x)}}^{(\theta p)'}\,dx
      \bigg)^{\frac 1{(\theta p)'}}
  \end{align*}
  where we used the selfadjointness of~$I_1$ and  Hölder's inequality . Now,
  condition~\eqref{eq:DreDur}, our assumption
  $\theta p \geq \max \set{1,\frac{np}{n+p}}$ and~\cite[Theorem~4]{MucWhe74} give
  \begin{align*}
    \bigg(\dashint_B \biggabs{\frac{I_1 (\indicator_B
    g)(x)}{\omega(x)}}^{(\theta p)'}\,dx
    \bigg)^{\frac 1{(\theta p)'}}
    &\lesssim 
    \bigg(\dashint_B \biggabs{\frac{g(x)}{\omega(x)}}^{p'}\,dx
    \bigg)^{\frac 1{p'}} \leq 1.
  \end{align*}
  This and the previous estimate shows
  \begin{align*}
    \dashint_B \abs{(v(x) - \mean{v}_{B})g(x)}\,dx
    &\lesssim  \bigg(\dashint_B (\abs{\nabla v}\, \omega)^{\theta p} \,dx\bigg)^{\frac
      1{\theta p}}.
  \end{align*}
  Taking the supremum over all admissible~$g$ proves the claim.
\end{proof}

\subsection{John-Nirenberg-Type Estimates}
\label{sec:john-nirenberg-type}

We  present several estimates of John-Nirenberg type for matrix-valued and
scalar weights in terms of its logarithm.

For a ball~$B_R$ with radius~$R$ we define the local~$\setBMO(B_R)$
space as the set of function~$f \in L^1(B_R)$ such that the semi-norm
\begin{align}
  \label{eq:ref-local-BMO}
  \abs{f}_{\setBMO(B_R)} &:=\sup_{\substack{0<r\le
                           R \\ x \in
  B_R}}\bigg(\frac{1}{\abs{B_r(x)}}\int_{B_r(x)  \cap B_R}
  \abs{f(x)-\mean{f}_{B_r(x)}}\, dx\bigg)
\end{align}
is finite.

First we will show that the $\setBMO$-estimates for the matrix-valued weight
$\log\bbM$ transfers to scalar weight $\log \omega$.
\begin{lemma}
  \label{lem:M-omega-BMO}
  For a matrix-valued weight~$\bbM$ and $\omega= \abs{\bbM}$ there holds
  \begin{align}
    \label{eq:M-omega-BMO-b}
    \dashint_B \abs{\log \omega(x) - \mean{\log \omega}_B}\,dx &\leq 2\,
    \dashint_B \abs{\log \bbM(x) - \mean{\log \bbM}_B}\,dx .
  \end{align}
  Moreover, $\abs{\log \omega}_{\setBMO(B)} \leq 2\, \abs{\log \bbM}_{\setBMO(B)}$.
\end{lemma}
\begin{proof}
  Let us abbreviate $\bbH(x) := \log \bbM(x)$ and
  $\bbH(y) := \log \bbM(y)$. For $\bbX \in \Rnnsym$ let $\mu(\bbX)$
  denote\footnote{ $\mu(\bbX)$ is just the logarithmic norm induced by
    the euclidean norm~$\abs{\cdot}$.}  the maximal eigenvalue of
  $\bbX \in \Rnnsym$. Then $\mu$ is sub-additive. As a consequence,
  \begin{align}
    \label{eq:M-omega-mu-subadd}
    \bigabs{\mu\big(\bbH(x)\big) - \mu\big(\bbH(y)\big)}
    &\leq \abs{\bbH(x)-\bbH(y)}.
  \end{align}
  Since~$\mu(\bbX) = \log\abs{\exp(\bbX)}$ for $\bbX \in \Rnnsym$, we
  have
  \begin{align*}
    \mu(\bbH(x)) = \log
    \abs{\exp(\bbH(x))} = \log \omega(x).
  \end{align*}
  Therefore, we can rewrite~\eqref{eq:M-omega-mu-subadd} as
  \begin{align}
    \label{eq:M-omega-BMO}
    \abs{\log(\omega(x)) - \log(\omega(y))}
    &\leq \abs{\bbH(x) -
      \bbH(y)} = \abs{\log \bbM(x) - \log \bbM(y)}.
  \end{align}
  This implies that
  \begin{align*}
    \dashint_B \abs{\log \omega(x) - \mean{\log \omega}_B}\,dx
    &\leq \dashint_B \dashint_B \abs{\log \omega(x) - \log
      \omega(y)}\,dy\,dx
    \\
    &\leq \dashint_B \dashint_B \abs{\log \bbM(x) - \log \bbM(y)}\,dy\,dx
    \\
    &\leq 2\,
    \dashint_B \abs{\log \bbM(x) - \mean{\log \bbM}_B}\,dx 
  \end{align*}
  using Jensen's inequality in the first step
  and~\eqref{eq:M-omega-BMO} in the second step.  This proves
   estimate~\eqref{eq:M-omega-BMO-b}. As a consequence $\abs{\log %
    \omega}_{\setBMO(\Rn)} \leq 2\, \abs{\log
    \bbM}_{\setBMO(\Rn)}$. The local version $\abs{\log
    \omega}_{\setBMO(B)} \leq 2\, \abs{\log \bbM}_{\setBMO(B)}$
  follows by simple modifications.
\end{proof}

\begin{proposition}
  \label{pro:small}
  There  exist  constants~$\kappa_1=\kappa_1(n,\Lambda)>0$ and
  $c_3>0$ such that the following holds.  If $q \geq 1$ and $\bbM$ is
  a matrix-valued weight with
  $\abs{\log \bbM}_{\setBMO(B)} \leq \frac{\kappa_1}{q}$, then
  \begin{align*}
    \Bigg(\dashint_{B}\bigg(\frac{
    |\bbM-\logmean{\bbM}_B|}{\abs{\logmean{\bbM}_{B}}}\bigg)^q dx
    \bigg)^{\frac 1q}
    &\le c_3\, % \kappa_1.
    q \abs{\log \bbM}_{\setBMO(B)}.
  \end{align*}
  The same holds with~$\bbM$ replaced by a scalar weight~$\omega$.
\end{proposition}
\begin{proof}
  Let us abbreviate  $\bbH(x) := \log \bbM(x)$. Then we have
  $\bbM(x) = \exp(\bbH(x))$ and
  $\logmean{\bbM}_B = \exp(\mean{\bbH}_B)$ and
  \begin{align*}
    \textrm{I} &:=     \Bigg(\dashint_{B}\bigg(\frac{
                 |\bbM-\logmean{\bbM}_B|}{\abs{\logmean{\bbM}_{B}}}\bigg)^q dx
                 \bigg)^{\frac 1q} =
                 \Bigg( \dashint_{B}\left(\frac{
                 |\exp(\bbH)-\exp(\mean{\bbH}_B)|}{\exp(\abs{\mean{\bbH}_{B}})}\right)^q\,dx
                 \Bigg)^{\frac 1q}
                 .
  \end{align*}
  Note that for all
  (hermetian) matrices~$\bbX,\bbY$ we have
  \begin{align*}
    \abs{\exp(\bbX+\bbY) - \exp(\bbX)} &\leq \abs{\bbY}
                                         \exp(\abs{\bbY})
                                         \exp(\abs{\bbX}), 
  \end{align*}
  Therefore, with $\bbX = \mean{\bbH}_B$ and $\bbY = \bbH - \mean{\bbH}_B$ we estimate
  \begin{align*}
    \textrm{I} 
    &\leq \Bigg( \dashint_{B}\big( \abs{\bbH - \mean{\bbH}_B}
      \exp(\abs{\bbH-\mean{\bbH}_B}) \big)^q\,dx \Bigg)^{\frac 1q}.
  \end{align*}
  So by H\"older's inequality
  \begin{align}
    \label{eq:small-aux1}
    \textrm{I} 
    &\leq \Bigg( \dashint_{B}\abs{\bbH - \mean{\bbH}_B}^{2q}\,dx
      \Bigg)^{\frac 1{2q}} \cdot 
      \Bigg( \dashint_{B}
      \exp(2q \abs{\bbH-\mean{\bbH}_B})\,dx \Bigg)^{\frac 1{2q}}.
  \end{align}
  It follows from the classical John-Nirenberg estimate in the form of~\cite[Corollary~3.1.8]{Gra14modern} that
  \begin{align}
    \label{eq:small-aux2}
    \Bigg( \dashint_{B}\abs{\bbH - \mean{\bbH}_B}^{2q}\,dx
      \Bigg)^{\frac 1{2q}} \leq c\,(c\,q!)^{\frac 1{2q}}
    \abs{\bbH}_{\setBMO(B)} \leq c\, q  \abs{\bbH}_{\setBMO(B)},
  \end{align}
  where we have used Stirling's formula in the last step.
  Another consequence of the John-Nirenberg estimate in the form
  of~\cite[Corollary~3.1.7]{Gra14modern} is that there
  exists~$\kappa_1 >0$ such that
  $q \abs{\bbH}_{\setBMO(B)} \leq \kappa_1$ implies
  \begin{align}
    \label{eq:small-aux3}
    \Bigg( \dashint_{B}
    \exp(2q \abs{\bbH-\mean{\bbH}_B})\,dx \Bigg)^{\frac 1{2q}}
    &\leq c^{\frac 1{2q}} \leq c.
  \end{align}
  Note that the results in~\cite{Gra14modern} are stated for
  $\setBMO(\Rn)$, but a simple extension from $\setBMO(B)$ to
  $\setBMO(\Rn)$ allows to deduce the local version. Moreover, the
  estimates for the vector valued~$\setBMO$ follow immediately from
  the scalar valued ones.

  Now, the claim follows from~\eqref{eq:small-aux1},
  \eqref{eq:small-aux2} and \eqref{eq:small-aux3}.
\end{proof}
We will now apply Proposition~\ref{pro:small} to deduce certain
properties for scalar weights.
\begin{proposition}
  \label{pro:small-scalar}
  There exists a constant~$\gamma>0$ such that the following holds for
  all weights~$\omega$.
  \begin{enumerate}
  \item
    \label{itm:small-scalar1}
    If $\abs{\log \omega}_{\setBMO(B)} \leq \frac{\gamma}{s}$ with
    $s \geq 1$, then
    \begin{align*}
    \bigg(\dashint_B \omega^s \,dx \bigg)^{\frac 1s}
      &\le 2\, \logmean{\omega}_B
    \end{align*}
    \item
    \label{itm:small-scalar1_inv}
    If $\abs{\log \omega}_{\setBMO(B)} \leq \frac{\gamma}{s}$ with
    $s \geq 1$, then
    \begin{align*}
    \bigg(\dashint_B \omega^{-s} \,dx \bigg)^{\frac 1s}
      &\le 2\, \frac{1}{\logmean{\omega}_B}.
    \end{align*}
  \item \label{itm:Ap1} If
    $\abs{\log \omega}_{\setBMO}\leq \gamma \min \set{\frac 1p, \frac
      1{p'}}$ with $1 < p < \infty$, then $\omega^p$ is an
    $\mathcal{A}_p$-Muckenhoupt weight and
    \begin{align*}
      [\omega^p]_{\mathcal{A}_p}^{\frac 1p} =  \sup_B \bigg(\dashint_B
      \omega^p \,dx \bigg)^{\frac 1p} 
      \bigg(\dashint_B \omega^{-p'} \,dx \bigg)^{\frac 1{p'}} &\leq 4.
    \end{align*}
  \item \label{itm:Ap2} Let $1 < p < \infty$ and $\theta \in (0,1)$
    such that $\theta p > 1$. If $\abs{\log \omega}_{\setBMO}\leq \gamma \min \set{\frac
      1p, 1-\frac 1{\theta p}}$, then
    \begin{align*}
      \bigg(\dashint_B \omega^p \,dx \bigg)^{\frac 1p}
      \bigg(\dashint_B \omega^{-(\theta p)'} \,dx \bigg)^{\frac
      1{(\theta p)'}} &\leq 4.
    \end{align*}
    (This is the ensures that~\eqref{eq:DreDur} in
    Proposition~\ref{pro:DreDur} holds.)
  \end{enumerate}
\end{proposition}
\begin{proof}
  We begin with~\ref{itm:small-scalar1}.  Let $\kappa_1$ and~$c_3$ be
  as in Proposition~\ref{pro:small}. We define $\gamma := \min
  \set{\kappa_1, 1/c_3}$. Now, assume that $\abs{\log
    \omega}_{\setBMO(B)} \leq \frac \gamma s$. Then it follows with
  Proposition~\ref{pro:small} that
  \begin{align*}
    \bigg(\dashint_B \omega^s \,dx \bigg)^{\frac 1s}
    &\leq
      \bigg(\dashint_B \abs{\omega-\logmean{\omega}_B}^s \,dx
      \bigg)^{\frac 1s} + \bigabs{\logmean{\omega}_B}
    \\
    &\leq 
      \logmean{\omega}_B
      \big( c_3\, s\,\abs{\log \omega}_{\setBMO(B)}
      +1\big)
      \\
    &\leq 
      2\,\logmean{\omega}_B.
  \end{align*}
  This proves~\ref{itm:small-scalar1}.

  Now, \ref{itm:small-scalar1_inv} is just~\ref{itm:small-scalar1}
  applied to~$\frac 1 \omega$ using also
that  $\logmean{\omega^{-1}}_B = (\logmean{\omega}_B)^{-1}$.
  
  Let us now prove~\ref{itm:Ap1}. If follows
  from~\ref{itm:small-scalar1} applied to~$\omega$ and~$p$,
  resp. $1/\omega$ and~$p'$, that
  \begin{align*}
    \bigg( \dashint_B \omega^p \,dx \bigg)^{\frac 1p}
    \bigg(\dashint_B \omega^{-p'} \,dx \bigg)^{\frac 1{p'}}
    &\leq 2
      \logmean{\omega}_B\cdot 2 \logmean{1/\omega}_B = 4.
  \end{align*}
  The proof of~\ref{itm:Ap2} is analogous to the one of~\ref{itm:Ap1}.
\end{proof}
\begin{remark}
  \label{rem:inverse}
  Since $\log(\bbM^{-1}) = -\log(\bbM)$ and $\log(\omega^{-1}) =-
  \log(\omega)$ for weights~$\bbM$ and~$\omega$, it is possible to
  apply Proposition~\ref{pro:small} and
  Proposition~\ref{pro:small-scalar} to~$\bbM^{-1}$ and~$\omega^{-1}$.
\end{remark}

%% ------------------------------------------------------------
\section{\Calderon{}-Zygmund Estimates}
\label{sec:cald-zygm-estim}
%% ------------------------------------------------------------

In this section we develop the full higher integrability result for
the solutions of our weighted $p$-Laplace equation.
Let~$\Omega \subset \Rn$ be a domain with Lipschitz boundary
and~$1<p<\infty$. Let~$\bbM$ be a matrix-valued weight on~$\Rn$ with
uniformly bounded condition number, i.e.~\eqref{eq:ass-M} holds. Since
$\bbM$ is symmetric and positive definite~\eqref{eq:ass-M} is
in fact equivalent to
\begin{align}
  \label{eq:mon-Mxi}
  \Lambda^{-1} \omega(x)\, \abs{\xi} \leq \abs{\bbM \xi} \leq
  \omega(x)\, \abs{\xi} \qquad \text{for all $\xi \in \Rn$}
\end{align}
and also
\begin{align}
  \label{eq:mon-M}
  \Lambda^{-1} \omega(x) \identity \leq \bbM(x) \leq \omega(x) \identity
\end{align}
both for all~$x \in \Omega$.

We assume in the following that the logarithmic weight~$\log \bbM$ has
small~$\setBMO$-norm, i.e.
\begin{align}
  \label{eq:smallness2}
  \abs{\log \bbM}_{\setBMO(\Omega)} \leq \kappa.
\end{align}
Hence, by Lemma~\ref{lem:M-omega-BMO} we have
$\abs{\log \omega}_{\setBMO(\Omega)} \leq 2 \kappa$.

Note that we do keep track of the dependence of the constants in terms
of~$\Lambda$ but we keep track of the dependence on~$\kappa$.

We assume that~$\kappa$ is so small that by
Proposition~\ref{pro:small-scalar} $\omega^p$ is an
$\mathcal{A}_p$-Muckenhoupt weight. In particular, smooth functions
are dense in~$W^{1,p}_\omega(\Omega)$. 

In the following let~$u \in W^{1,p}_\omega(\Omega)$ be a weak solution
of~\eqref{eq:sysM} with $G \in L^p_\omega(\Omega)$, i.e.
\begin{align}
  \label{eq:weak}
  \int_\Omega \abs{\bbM \nabla u}^{p-2} \bbM^2 \nabla u \cdot
  \nabla \xi\,dx
  &=
  \int_\Omega \abs{\bbM G}^{p-2} \bbM^2 G \cdot
  \nabla \xi\,dx
\end{align}
for all $\xi \in C^\infty_0(\Omega)$ or equivalently all
$\xi \in W^{1,p}_{0,\omega}(\Omega)$. Note that the existence of a
weak solution is ensured by standard arguments from the calculus of
variations, since~$\omega^p$ is an~$\mathcal{A}_p$-Muckenhoupt weight.

%% ------------------------------------------------------------

%% ------------------------------------------------------------
\subsection{Caccioppoli Estimate and Reverse H\"older's Inequality}
\label{ssec:caccioppoli-estimate}
%% ------------------------------------------------------------

We begin with the standard Caccioppoli estimates.
\begin{proposition}[Caccioppoli]
  \label{pro:caccioppoli}
  For all balls~$B$ with $2B \subset \Omega$ there holds
  \begin{align*}
    \dashint_B |\nabla u|^p \omega^p\,dx
    &\le c\, \dashint_{2B}\biggabs{\frac{u -
      \mean{u}_{2B}}{r_B}}^p \omega^p \,dx +c\,\dashint_{2B} |G|^{p} \omega^p
      \,dx.
  \end{align*}
\end{proposition}
\begin{proof}
  Fix a smooth cut-off function $\eta$ with
  $\indicator_B \leq \eta \leq \indicator_{2B}$ and
  $|\nabla\eta|\le \frac{c}{r_B}$. Using the test
  function~$\eta^p (u-\mean{u}_{2B})$ in  \eqref{eq:weak}  we get
  \begin{align*}
    \int \!\! |\bbM \nabla u|^{p-2} \bbM
    \nabla u \cdot \bbM \nabla(\eta^p(u\!-\!\mean{u}_{2B}) \,dx
    &= \!\!
    \int \!\! |\bbM G|^{p-2}
    \bbM G \cdot \bbM \nabla(\eta^p(u\!-\!\mean{u}_{2B})\,dx.
  \end{align*}
  Using~\eqref{eq:mon-Mxi} we obtain by standard calculations
  \begin{align*}
    \int_{2B}\eta^p |\nabla u|^p \omega^p\,dx
    &\leq c\,\int_{2B}\eta^{p-1} |\nabla u|^{p-1}
      \biggabs{ \frac{u-\mean{u}_{2B}}{r_B}} \,\omega^p\,dx
    \\
    &\quad +c\,\int_{2B}\eta^p|G|^{p-1}
      \abs{\nabla u} \,\omega^p\,dx
    \\
    &\quad +c\,\int_{2B}\eta^{p-1}|G|^{p-1}       \biggabs{
      \frac{u-\mean{u}_{2B}}{r_B}}  \omega^p \,dx.
  \end{align*}
  We use Young's inequality, absorb the term with $\eta^p \abs{\nabla
    u}^p \omega^p$ and obtain
  \begin{align*}
    \int_B|\nabla u|^{p} \omega^p\,dx
    &\le c\, \int_{2B}\biggabs{\frac{u -
      \mean{u}_{2B}}{r_B}}^p \omega^p \,dx +c\,\int_{2B} |G|^{p} \omega^p
      \,dx.
  \end{align*}
  This proves the claim.
\end{proof}
From the Caccioppoli estimate we derive as usual the reverse H\"older
estimate.
\begin{proposition}
  \label{pro:reverse}
  There exists~$\kappa_2 >0$ and~$\theta \in (0,1)$ such that for all
  balls~$B$ with~$2B \subset \Omega$ there holds: if
  $\abs{\log \bbM}_{\setBMO(2B)} \leq \kappa_2=\kappa_2(p,n,\Lambda)$,
  then
  \begin{align*}
    \dashint_B \abs{\nabla u}^p \omega^p\,dx
    &\lesssim \, \bigg( \dashint_{2B} \abs{\nabla u}^{\theta p}\omega^p \,dx
      \bigg)^{\frac 1\theta}  +\dashint_{2B} |G|^{p} \omega^p
      \,dx.
  \end{align*}
\end{proposition}
\begin{proof}
  We can choose~$\kappa_2$ so small such that
  Proposition~\ref{pro:small-scalar}~\ref{itm:Ap2} ensures the
  applicability of the weighted \Poincare{}-Estimate of
  Proposition~\ref{pro:DreDur}. This and the Caccioppoli estimate of
  Proposition~\ref{pro:caccioppoli} prove the claim.
\end{proof}
An application of Gehring's lemma (e.g. \cite[Theorem~6.6]{Giu03})
immediately gives the following consequence.
\begin{corollary}[Small Higher Integrability]
  \label{cor:small-high-int}
  There exists~$\kappa_2 >0$ and~$s>1$ such that for all
  balls~$B$ with~$2B \subset \Omega$ there holds: if
  $\abs{\log \bbM}_{\setBMO(2B)} \leq \kappa_2=\kappa_2(p,n,\Lambda)$,
  then
  \begin{align*}
    \bigg(\dashint_B \big(\abs{\nabla u}^p \omega^p\big)^s\,dx \bigg)^{\frac 1s}
    &\lesssim \, \dashint_{2B} \abs{\nabla u}^p\omega^p \,dx
      +\bigg(\dashint_{2B} \big(|G|^p \omega^p\big)^s
      \,dx \bigg)^{\frac 1s}.
  \end{align*}
\end{corollary}

\subsection{Interlude on Orlicz Functions}
\label{sec:interl-orlicz-funct}

For the analysis in the following sections it is useful to introduce a
few auxiliary functions and some basic properties on Orlicz functions. The
N-function 
\begin{align*}
  \phi(t) &:= \tfrac 1p t^p
\end{align*}
is the natural one for our problems.

Then
\begin{align}
  A(\xi) &= \abs{\xi}^{p-2} \xi = \frac{\phi'(\abs{\xi})}{\abs{\xi}} \xi.
\end{align}
Let us define
\begin{align*}
  V(\xi) &:=\sqrt{\frac{\phi'(\abs{\xi})}{\abs{\xi}}} \xi = |\xi|^{\frac{p-2}{2}}\xi.
\end{align*}
In general a function $\psi:\setR_{\ge 0}\to\setR$ is called an
N-function if and only if there is a right-continuous, positive on the
positive real line, and non-decreasing function
$\psi':\setR_{\ge 0}\to\setR$ with $\psi'(0)=0$ and
$\lim_{t \to \infty} \psi'(t)=\infty$ such that
$\psi(t) = \int_0^{t}\psi'(\tau)\,d \tau$. An N-function is said to
satisfy the $\Delta_2$-condition if and only if there is a constant
$c>1$ such that $\psi(2t)\le c\,\psi(t)$.

The conjugate of an N-function~$\psi$ is defined as
\begin{align*}
  \psi^*(t) := \sup_{s \geq 0} \left( ts - \psi(s) \right), \quad t \geq 0.
\end{align*}
In our case~$\phi^*(t) = \tfrac{1}{p'} t^{p'}$.

Moreover, we need the notion of shifted N-functions first
introduced in~\cite{DieEtt08}. Here, we use the slight variant
of~\cite[Appendix~B]{DieForTomWan19} with even nicer properties.

We define the shifted N-functions $\phi_a$ for $a \geq 0$ by
\begin{align}
  \phi_a(t) := \int_0^t \frac{\phi'(a \vee s)}{a \vee s} s \, ds,
\end{align}
where $s_1 \vee s_2 := \max \{ s_1, s_2 \}$ for $s_1, s_2 \in
\setR$. Then\footnote{The version from~\cite{DieEtt08} used~$+$
  instead of~$\vee$. This implies that the equality
  in~\eqref{eq:dual-shift} has to be replaced by~$\eqsim$. This would
  still be sufficient for the purpose of this paper.}
\begin{align}
  \label{eq:phi_a-approx}
  \begin{aligned}
    \phi_a(t) &\eqsim (a \vee t)^{p-2} t^2,
    \\
    \phi_a'(t) &\eqsim (a \vee t)^{p-2} t,
  \end{aligned}
\end{align}
with constants only depending on~$p$. The index~$a$ is called
the~\emph{shift}.  Obviously, $\phi_0= \phi$. Moreover, if
$a \eqsim b$, then $\phi_a(t) \eqsim \phi_b(t)$.
For the shifted
N-functions~$\phi_a$ we have
\begin{align}
  \label{eq:dual-shift}
  (\phi_a)^* &= (\phi^*)_{\phi'(a)}.
\end{align}
Thus, we get the useful equation
\begin{align}
  \label{eq:dual-shift2}
  (\phi_{\abs{\xi}})^* &= (\phi^*)_{\abs{A(\xi)}}.
\end{align}
Moreover, the family~$\phi_a$, $a\geq 0$, as
well as its conjugate functions also satisfy the~$\Delta_2$-condition
with a $\Delta_2$-constant uniformly bounded with respect to~$a$. In
particular, we can apply Young's inequality to obtain: for
every~$\delta>0$ there exists~$c_\delta=c_\delta(\delta,p)\geq 1$ such
that for all~$s,t,a\geq 0$
\begin{align}
  \label{eq:young}
  s\,t &\leq c_\delta (\phi_a)^*(s) + \delta\,\phi_{a}(t).
\end{align}
Using~$\phi_a(t) \eqsim \phi_a'(t)\,t$ and $(\phi_a)^*\eqsim (t\phi_a'(t))$
we get the following equivalent versions
\begin{align}
  \label{eq:young2}
  \begin{aligned}
    \phi_a'(s)\,t &\leq c_\delta \phi_a(s) + \delta\,\phi_a(t),
    \\
    \phi_a'(s)\,t &\leq \delta \phi_a(s) + c_\delta\,\phi_{a}(t)
  \end{aligned}
\end{align}
for $s,t,a \geq 0$.

Moreover, the following simple equivalence holds for~$a \geq 0$
\begin{align}
  \label{eq:phialambdaa}
  \phi_a(\lambda a)&\eqsim\begin{cases} \lambda^2 \phi(a),&\quad
    \text{for }
    \lambda<1,
    \\
    \phi(\lambda a) &\quad \text{for }\lambda>1.
 \end{cases}
\end{align}

The important relation between~$A$, $V$ and the~$\phi_a$ is best
summarized in the following lemma.
\begin{lemma}[{\cite[Lemma~41]{DieForTomWan19}}]
  \label{lem:hammer}
  For all~$P,Q \in \Rn$ there holds
  \begin{align*}
    \big( A(P)-A(Q)  \big) \cdot ( P-Q )
    & \eqsim \abs{V(P)-V(Q)}^2
    \\
    & \eqsim \phi_{\abs{Q}} \left( \abs{P-Q} \right)
    \\
    & \eqsim (\phi^*)_{\abs{A(Q)}} \left( \abs{A(P)-A(Q)}
      \right).
    \\
    \intertext{and}
    A(Q) \cdot Q = \abs{V(Q)}^2
    &\eqsim
      \phi_{\abs{Q}}(\abs{Q}) \eqsim \phi(\abs{Q})
    \\
    \intertext{and}
    \abs{A(P)-A(Q)} &\eqsim (\phi_{\abs{Q}})'(\abs{P-Q}).
  \end{align*}
  The implicit constants depend only on $p$.
\end{lemma}
Also of strong use is the possibility to change the shift:
\begin{lemma}[{Change of shift, \cite[Corollary~44]{DieForTomWan19}}]
  \label{lem:change_of_shift}
  For~$\delta>0$ there exists~$c_\delta=c_\delta(p)$ such that for
  all~$P,Q \in \Rn$ there holds
  \begin{align*}
    \phi_{|P|}(t) & \leq c_{\delta} \phi_{|Q|}(t) + \delta\,
                    \abs{V(P)-V(Q)}^2,
    \\
    (\phi_{|P|})^*(t) & \leq c_{\delta} (\phi_{|Q|})^*(t) + \delta\,
                        \abs{V(P)-V(Q)}^2.
    \\
  \end{align*}
  The implicit constants depend only on $p$.
\end{lemma}
We are in particular interested in the following special case for~$Q=0$.
\begin{lemma}[Removal of shift]
  \label{lem:removal-shift}
  For all $a \in \Rn$, all~$t \geq 0$ and all~$\delta \in (0,1]$ there holds
  \begin{align}
    \label{eq:removal-shift1}
    \phi_{\abs{a}}'(t) &\leq \phi'\bigg(\frac t \delta \bigg) \vee
                         \big( \delta \phi'(\abs{a})\big),
                         \\
    \label{eq:removal-shift2}
    \phi_{\abs{a}}(t) &\leq \delta\, \phi(\abs{a}) + c\, \delta\,
                        \phi\bigg(\frac{t}{\delta}\bigg),
    \\
    \label{eq:removal-shift3}
    (\phi_{\abs{a}})^*(t) &\leq \delta\, \phi(\abs{a}) + c\, \delta\,
                        \phi^*\bigg(\frac{t}{\delta}\bigg),
  \end{align}
  where~$c$ only depends on the~$\Delta_2$ constants of~$\phi$ and~$\phi^*$.
\end{lemma}
\begin{proof}
  We start with~\eqref{eq:removal-shift1}.  If $t \leq \delta \abs{a}$, then
  \begin{align*}
    \phi_{\abs{a}}'(t) &= \frac{\phi'(\abs{a} \vee t)}{\abs{a} \vee t}
                         t = \frac{\phi'(\abs{a})}{\abs{a}} t \leq \delta\,\phi'(\abs{a}).
  \end{align*}
  If $t \geq \delta \abs{a}$, then with $0 < \delta \leq 1$
  \begin{align*}
    \phi_{\abs{a}}'(t) & = \frac{\phi'(\abs{a} \vee t)}{\abs{a} \vee
                         t} t \leq \phi'(\abs{a} \vee t) \leq
                         \phi'\bigg(\frac t \delta \bigg). 
  \end{align*}
  This proves~\eqref{eq:removal-shift1}.
  We continue with~\eqref{eq:removal-shift2}.
  From~Lemma~43 of~\cite{DieForTomWan19} (with $b=0$) we have
  \begin{align*}
    \phi_{\abs{a}}'(t) &\leq c\, \phi'(\abs{a}).
  \end{align*}
  Thus, for~$\delta>0$ we obtain with Young's inequality
  \begin{align*}
    \phi_{\abs{a}}(t)
    &\leq \phi_{\abs{a}}'(t)\,t
    \\
    &\leq
      c\, \phi'(\abs{a})\,t
    \\
    &\leq \delta \phi^*\big( \phi'(\abs{a})\big) + c\, \delta\, \phi(t/\delta)
    \\
    &\leq \delta\, c\, \phi(\abs{a}) + c\, \delta\,\phi(t/\delta).
  \end{align*}
  This proves the claim (if we replace~$\delta$ by smaller one). Now,
  \eqref{eq:removal-shift3} follows from \eqref{eq:removal-shift2}
  using $(\phi_{\abs{a}})^* = (\phi^*)_{\phi'(\abs{a})}$ and the
  equivalence $\phi^*(\phi'(\abs{a})) \eqsim \phi(\abs{a})$.
\end{proof}
%\begin{remark}[Removal of shift]
%   We will in particular use the special case~$Q=0$ in
%   Lemma~\ref{lem:change_of_shift}. In this case we can rewrite the
%   estimate as 
%   \begin{align*}
%     \phi_{\abs{P}}(t) &\leq c_\delta \phi(t) + \delta\, \phi(\abs{P}).
%   \end{align*}
%   A careful tracking of the dependency of~$c_\delta$ on on~$\delta$
%   reveals that
%   \begin{align*}
%     \phi_{\abs{P}}(t) &\leq c\, \delta^{1-p} \phi(t) + \delta\, \phi(\abs{P}).
%   \end{align*}
%   for all $t\geq 0$, all $P \in \Rn$ and all~$\delta>0$.
% \end{remark}

\subsection{Comparison Estimate}
\label{sec:comparison-estimate}

To obtain higher integrability beyond
Corollary~\ref{cor:small-high-int} we need to derive comparison
estimates. This is where we need the logarithmic mean~$\logmean{\bbM}_B$ of
our matrix-values weight~$\bbM$, which is a positive definite matrix. 

Recall, that
$\mathcal{A}(x,\xi) := |\bbM(x)\xi|^{p-2} \bbM^2 \xi = \bbM A(\bbM
\xi)$ with $A(\xi) := \abs{\xi}^{p-2} \xi$. Now, for a
ball~$B \subset \Omega$ we abbreviate
\begin{align}
  \begin{aligned}
    \bbM_B &:= \logmean{\bbM}_B,
    \\
    \omega_B &:= \logmean{\omega}_B
  \end{aligned}
\end{align}
and define
\begin{align}
  \label{eq:def-AB}
  \mathcal{A}_B(\xi) &:=  | \bbM_B\xi|^{p-2}
                       \bbM_B^2\xi =  \bbM_B
                       A(\bbM_B \xi).
\end{align}
We will use~$\mathcal{A}_B$ below to define a suitable comparison
problem. It naturally arises if we minimize the energy
$\int \tfrac 1p \abs{\bbM_B \nabla h}^p\,dx$.

We abbreviate 
\begin{align*}
  \mathcal{A}(x,\xi) &= \bbM(x) A(\bbM(x) \xi),
  \\
  \mathcal{V}(x,\xi) &= V(\bbM(x) \xi),
  \\
  \mathcal{V}_B(\xi) &= V(\bbM_B \xi).
\end{align*}
Then we have 
\begin{align*}
  \mathcal{A}(x,\xi) \cdot \xi
  &= \abs{\mathcal{V}(x,\xi)}^2,
  \\
  \mathcal{A}_B(\xi) \cdot \xi
  &= \abs{\mathcal{V}_B(\xi)}^2,
  \\
  \intertext{and}
  \abs{\mathcal{A}_B(\xi)}
  &\lesssim
    \omega_B^p \abs{\xi}^{p-1}.
\end{align*}
We will now compare~$u$ locally to its $\mathcal{A}_B$-harmonic
counterpart. Due to some later localization argument will not compare
directly with~$u$ but with to a truncated version of it. This
technique goes back to Kinnunen and Zhou~\cite{KinZho99}.
Originally, we wanted to compare directly with~$u$, since this seemed
us more natural to us. However, we encountered problems with the
localization of the maximal operators later and decided to proceed as
Kinnunen and Zhou in~\cite{KinZho99}.

We fix a ball $B_0:= B_R(x_0)$ with $2B_0 \subset \Omega$ and choose a
cut-off function $\zeta \in C^\infty_0(B_0)$ with
$\indicator_{\frac 12 B_0} \leq \zeta \leq \indicator_{B_0}$ and
$\norm{\nabla \zeta}_\infty \leq c R^{-1}$. Let us define the localized
function
\begin{align}
  \label{eq:def-z}
  z &:= (u - \mean{u}_{2B_0}) \zeta^{p'}.
\end{align}
Moreover, we will use
\begin{align}
  \label{eq:def-g}
  g &:= \zeta^{p'} \nabla u - \nabla z =  -(u-\mean{u}_{2B_0})\,\nabla
      (\zeta^{p'})  = -(u - \mean{u}_{2B_0})
      p' \zeta^{p'-1} \nabla \zeta. 
\end{align}
Now, let $B=B_r$ denote an arbitrary ball with $4B \subset 2B_0$.  We
want to compare our localized function~$z$ on the ball~$B$ with the
weak solution~$h$ of
\begin{align}
  \label{eq:h_sol}
  \begin{alignedat}{2}
    -\divergence \big(\mathcal{A}_B(\nabla h)\big)&=0
    &\qquad&\text{in }B,
    \\
    h&=z &&\text{on } \partial B.
  \end{alignedat}
\end{align}
The natural function space for~$h$ is
$W^{1,p}_{\omega_B}(B)$ and~$h$ is the unique minimizer of
\begin{align}
  \label{eq:min-h}
  v \mapsto \int_B \phi(\bbM_B \abs{\nabla v})\,dx
\end{align}
subject to the boundary condition~$v=z$ on~$\partial B$.  We will
explain the well posedness of the boundary conditions below in
Lemma~\ref{lem:u-omegaB}.

Recall  that
\begin{align*}
  \Lambda^{-1} \omega(x) \identity \leq \bbM(x) \leq \omega(x) \identity.
\end{align*}
Considering the eigenvalues it follows that
\begin{align*}
  \log(\Lambda^{-1} \omega(x)) \identity \leq \log \bbM(x)  \leq
  (\log \omega)\, \identity.
\end{align*}
This also follows from the operator monotonicity of the matrix
logarithm, see the survey article of
Chansangiam~\cite[Example~13]{Cha15}. Taking the mean value we obtain
\begin{align*}
  \Lambda^{-1} \exp\big(\mean{\log \omega}_B\big) \identity \leq
  \exp\big(\mean{\log \bbM}_B\big)  \leq 
  \exp\big(\mean{\log \omega}_B\big) \identity.
\end{align*}
Comparing again the eigenvalues we obtain by taking the
exponential\footnote{Note that the matrix exponential is not operator
  monotone on~$\Rnnsym$. However, we compare here only with a
  multiple of the identity matrix.}
\begin{align*}
  \big(\log(\Lambda^{-1}) + \mean{\log \omega}_B\big) \identity \leq
  \mean{\log \bbM}_B  \leq 
  \mean{\log \omega}_B \identity.
\end{align*}
In other words,
\begin{align}
  \label{eq:mon-MB}
  \Lambda^{-1} \omega_B \identity \leq \bbM_B \leq
  \omega_B \identity
\end{align}
and therefore
\begin{align}
  \label{eq:mon-MBxi}
  \Lambda^{-1} \omega_B\, \abs{\xi} &\leq \abs{\bbM_B \xi} \leq
  \omega_B\, \abs{\xi} \qquad \text{for all~$\xi\in \Rn$},
  \\
  \intertext{and}
  \label{eq:mon-MBest}
  \Lambda^{-1} \omega_B\, &\leq \abs{\bbM_B} \leq
  \omega_B.
\end{align}
For the well posedness of the boundary condition of the
equation~\eqref{eq:h_sol} it is necessary that~$u$ has enough
regularity. The following lemma ensures  that~$u$ has indeed the required
regularity natural to~\eqref{eq:h_sol}.
\begin{lemma}
  \label{lem:u-omegaB}
  There exists~$\kappa_3>0$ and~$s>1$ such that for all balls~$B$
  with~$2B \subset \Omega$ there holds: if
  $\abs{\log \bbM}_{\setBMO(4B)} \leq \kappa_3=
  \kappa_3(p,n,\Lambda)$, then 
  \begin{align*}
    \dashint_B \abs{\nabla u}^p \omega_B^p\,dx
    &\lesssim \, \dashint_{2B} \abs{\nabla u}^p\omega^p \,dx
      +\bigg(\dashint_{2B} \big(|G|^p \omega^p\big)^s
      \,dx \bigg)^{\frac 1s}.
  \end{align*}

\end{lemma}
\begin{proof}
  By H\"older's inequality and Proposition~\ref{pro:small-scalar}
  \begin{align*}
    \dashint_B \abs{\nabla u}^p \omega_B^p\,dx
    &\leq  \bigg(\dashint_B \big(\abs{\nabla u}^p
      \omega^p\big)^s\,dx \bigg)^{\frac 1s}
      \bigg( \dashint_B \big(\omega_B \omega^{-1}
      \big)^{ps'}\,dx \bigg)^{\frac 1{s'}}
    \\
    &\leq  \bigg(\dashint_B \big(\abs{\nabla u}^p
      \omega^p\big)^s\,dx \bigg)^{\frac 1s}
      2^p.
  \end{align*}
  Now, the claim follows with Corollary~\ref{cor:small-high-int}.
\end{proof}
It follows from this lemma that $u \in W^{1,p}_{\omega_B}(B)$ (assuming
the smallness condition on~$\log \bbM$). Thus also $z=(u-\mean{u}_B) \zeta^{p'} \in
W^{1,p}_{\omega_B}(B)$. In particular, it follows that
equation~\eqref{eq:h_sol} is well posed with a unique
solution~$h \in W^{1,p}_{\omega_B}(B)$ and $h=z$
on~$\partial B$ in the usual trace sense.

The following proposition summarizes the higher regularity properties
of~$h$.
\begin{proposition}
  \label{pro:reg-h}
  Let~$h$ be the solution of~\eqref{eq:h_sol}. Then
  \begin{align*}
    \sup_{\frac 12 B} \abs{\nabla h}^p \omega_B^p &\leq c\,
    \dashint_B \abs{\nabla h}^p \omega_B^p\,dx.
  \end{align*}
  Moreover, there exists~$\alpha \in (0,1)$ such that for all
  $\lambda \in (0,1)$
  \begin{align*}
     \dashint_{\lambda B} \abs{\mathcal{V}_B(\nabla h) -
     \mean{\mathcal{V}_B(\nabla h)}_{\lambda B}}^2\,dx
    &\leq c\, \lambda^{2\alpha}
    \dashint_{B} \abs{\mathcal{V}_B(\nabla h) -
    \mean{\mathcal{V}_B(\nabla h)}_{B}}^2\,dx.
  \end{align*}
  The constants~$c,\alpha$ only depend on~$p$, $n$ and $\Lambda$.
\end{proposition}
\begin{proof}
  If~$\omega=1$, then both estimates just follow from Lemma~5.8 and
  Theorem~6.4 of~\cite{DieSV09}.  In the general case we have
  by~\eqref{eq:mon-MB}
  \begin{align*}
    \Lambda^{-1} \omega_B \identity \leq \bbM_B \leq \omega_B \identity.
  \end{align*}
  Since~\eqref{eq:h_sol} and our estimates scale by scalar factors, we
  can assume without loss of generality~$\omega_B=1$. We can also
  assume that~$B$ is centered at~$0$.
  
  Since $\bbM_B$ is also symmetric, we can find an orthogonal
  matrix~$Q$ and a diagonal matrix~$\bbD_B$ such that $\bbM_B = Q
  \bbD_B Q^*$. If we define $w(x) := h(Qx)$, then it follows that~$w$
  solves~\eqref{eq:h_sol} with~$\bbM_B$ replaced by~$\bbD_B$. The
  boundary values are also rotated but this is not important for our
  estimates.  Hence, we can assume without loss of generality
  that~$\bbM_B$ is already diagonal. We have
  reduced the claim so far to a diagonal matrix~$\bbM_B=\bbD_B$ with
  \begin{align*}
    \Lambda^{-1} \identity \leq \bbD_B \leq \identity.
  \end{align*}
  Since~$\Lambda$ is fixed we can use an anisotropic scaling
  $x \mapsto \bbD_B^{-1}x$.  This turns estimates on balls into
  estimates on ellipses of uniformly bounded eccentricity. Thus, we
  deduce from Lemma~5.8 and Theorem~6.4 of~\cite{DieSV09} that our
  estimates are valid for ellipses instead of balls.  Since all balls
  can be covered by slightly enlarged ellipses and vice versa, it
  follows that the estimates are also true for balls with slightly
  enlarged constants (depending on~$\Lambda$). This proves the claim. 
\end{proof}

The following lemma allows us to control the difference of the
mapping~$\mathcal{A}(\cdot,\cdot)$ and its frozen
version~$\mathcal{A}_B(\cdot)$.
\begin{lemma}
  \label{lem:AB-A}
  For all~$\xi \in \Rn$ and all~$x \in B$ there holds
  \begin{align*}
    \abs{\mathcal{A}_B(\xi) - \mathcal{A}(x,\xi)}
    &=
      \bigabs{\bbM_B A(\bbM_B \xi) - \bbM(x) A(\bbM(x) \xi)}
    \\
    &\lesssim
      \frac{\bigabs{\bbM_B-\bbM(x)}}{\abs{\bbM_B} + \abs{\bbM(x)}} \big(
      \abs{\mathcal{A}_B(\xi)} + \abs{\mathcal{A}(x,\xi)}\big).
  \end{align*}
\end{lemma}
\begin{proof}
  Note that
  $\abs{\bbM_B \xi} \eqsim \abs{\bbM_B}\abs{\xi}$
  and $\abs{\bbM \xi} \eqsim \abs{\bbM} \abs{\xi}$ due to
  \eqref{eq:mon-Mxi}, \eqref{eq:mon-MBxi}
  and~\eqref{eq:mon-MBest}. Thus, we can 
  estimate 
  \begin{align*}
    \lefteqn{\abs{\mathcal{A}_B(\xi) - \mathcal{A}(x,\xi)}
    =
    \bigabs{\bbM_B A(\bbM_B \xi) - \bbM(x) A(\bbM(x) \xi)}} \qquad
    &
    \\
    &\leq \bigabs{\bbM_B} \bigabs{A(\bbM_B \xi) - A(\bbM(x) \xi)} +
      \abs{\bbM_B - \bbM(x)}\, \abs{A(\bbM(x) \xi)}
    \\
    &\lesssim \bigabs{\bbM_B} \phi_{\abs{\bbM_B \xi} \vee \abs{\bbM(x) \xi}}'
      \big(\abs{\bbM_B \xi - \bbM(x) \xi}\big) +
      \abs{\bbM_B - \bbM(x)}\, \abs{A(\bbM(x) \xi)}
    \\
    &\lesssim \bigabs{\bbM_B} \frac{\abs{\bbM_B - \bbM(x)}}{\abs{\bbM_B} +
      \abs{\bbM(x)}}  \big( \big(\abs{\bbM_B} +
      \abs{\bbM(x)}\big) \abs{\xi} \big)^{p-1}
    \\
    &\quad +
      \abs{\bbM_B - \bbM(x)}\, \big(\abs{\bbM(x)} \abs{\xi}\big)^{p-1}
    \\
    &\lesssim \frac{\bigabs{\bbM_B-\bbM(x)}}{\abs{\bbM_B} + \abs{\bbM(x)}}
      \big(\abs{\bbM_B} + \abs{\bbM(x)}\big)^p \abs{\xi}^{p-1}
    \\
    &\eqsim
      \frac{\bigabs{\bbM_B-\bbM(x)}}{\abs{\bbM_B} + \abs{\bbM(x)}} \big(
      \abs{\mathcal{A}_B(\xi)} + \abs{\mathcal{A}(x,\xi)}\big)
  \end{align*}
  using also~\eqref{eq:phi_a-approx} in the fourth step.  This proves the lemma.
\end{proof}
We are now prepared for our comparison estimate.
From the equation~\eqref{eq:h_sol} for~$h$, the homogeneity
of~$\mathcal{A}$, and the equation~\eqref{eq:calA} for~$u$ we deduce
that
\begin{align}
  \label{eq:z}
  \begin{aligned}
    \lefteqn{-\divergence \big(\mathcal{A}_B(\nabla
      z)-\mathcal{A}_B(\nabla h)) } \quad &
    \\
    &= - \divergence\big(\mathcal{A}_B(\nabla z)\big)
    \\
    &= - \divergence\big(\mathcal{A}_B(\nabla z) - \mathcal{A}(\cdot,
    \nabla z)\big) - \divergence\big(\mathcal{A}(\cdot, \nabla z) -
    \mathcal{A}(\cdot,\nabla z + g))
    \\
    &\quad - \divergence\big(\mathcal{A}(\cdot,\zeta^{p'} \nabla u))
    \\
    &= - \divergence\big(\mathcal{A}_B(\nabla z) - \mathcal{A}(\cdot,
    \nabla z)\big) - \divergence\big(\mathcal{A}(\cdot, \nabla z) -
    \mathcal{A}(\cdot,\nabla z + g))
    \\
    &\quad - \divergence\big(\mathcal{A}(\cdot,\nabla u) \zeta^p)
    \\
    &= - \divergence\big(\mathcal{A}_B(\nabla z) -
    \mathcal{A}(\cdot,\nabla z)\big) -
    \divergence\big(\mathcal{A}(\cdot, \nabla z) - \mathcal{A}(\cdot,\nabla
    z + g))
    \\
    &\quad - \zeta^p\divergence\big(\mathcal{A}(\cdot, G)) - \nabla (\zeta^p)
    \mathcal{A}(\cdot,\nabla u).
  \end{aligned}
\end{align}
\begin{proposition}[Comparison]
  \label{pro:comparison}
  Recall, that $B=B_r$, $B_0=B_R(x_0)$, $4B \subset 2B_0$ and $z,g,h$ are given
  by~\eqref{eq:def-z}, \eqref{eq:def-g} and \eqref{eq:h_sol},
  respectively.
  There exist $s>1$ and $\kappa_4=\kappa_4(p,n,\Lambda)$, such that if
  $\abs{\log \bbM }_{\setBMO(2B)}\le \kappa_4$, then for
  every~$\delta \in (0,1)$ there holds
  \begin{align*}
    \lefteqn{\dashint_B  \abs{\mathcal V_B(\nabla
    h) -\mathcal V_B(\nabla z)}^2 \,dx} \qquad
    &
    \\
    &\leq c\, \big(\abs{\log
      \bbM}_{\setBMO(B)}^2 + \delta \big) 
      \bigg( \dashint_B (\abs{\nabla z}^p \omega^p)^s \, dx\bigg
      )^{\frac 1 {s}} 
    \\
    & + \quad      c\, \delta^{1-p} \Bigg(
      \dashint_{4B} \bigg(\frac{\abs{u-\mean{u}_{2B_0}}^p}{R^p}
      \omega^p \bigg)^s \,dx \Bigg)^{\frac 1s}
      + c \,  \delta^{1-p}
      \bigg( \dashint_{4B} \big(\zeta^p \abs{G}^{p} \omega^p\big)^s\,dx
      \bigg)^{\frac 1s}.
  \end{align*}
\end{proposition}
\begin{proof}
  Let $s>1$ be as in Corollary~\ref{cor:small-high-int} (so $s$ just
  depends on~$p$).  We assume
  that~$\abs{\log \bbM}_{\setBMO(B)} \leq \kappa_3$ with $\kappa_3$
  from Lemma~\ref{lem:u-omegaB} so that the comparison equation is
  well defined. We will add in the proof several other smallness
  conditions on~$\abs{\log \bbM}_{\setBMO(2B)}$ that will finally
  determine the value of~$\kappa_4$.

  Using~\eqref{eq:h_sol} and~\eqref{eq:calA} with the test
  function~$\abs{B}^{-1}(z-h)$ we obtain
  \begin{align*}
    \textrm{I}_0 &:= \dashint_B \big( \mathcal{A}_B(\nabla
                 z)-\mathcal{A}_B(\nabla h)\big) \cdot (\nabla  z-\nabla
                 h )\,dx
    \\
               &=      \dashint_B \big( \mathcal{A}_B(\nabla
                 z)-\mathcal{A}(x,\nabla z) \big) \cdot (\nabla z-\nabla
                 h )\,dx
    \\
               &\quad + \dashint_B \big( \mathcal{A}(x,\nabla
                 z)-\mathcal{A}(x,\nabla z+g) \big) \cdot (\nabla z-\nabla
                 h )\,dx
    \\
               &\quad +\dashint_B \zeta^p \mathcal{A}(x,G) \cdot (\nabla z-\nabla
                 h)\,dx 
    \\
               &\quad +\dashint_B \nabla (\zeta^p)
                 \mathcal{A}(x,\nabla u) (z-h)\,dx 
    \\
               &=: \textrm{I}_1 + \textrm{I}_2 + \textrm{I}_3 + \textrm{I}_4.
  \end{align*}
  Now,
  \begin{align}
    \label{eq:Iequiv}
    \textrm{I}_0 &\eqsim \dashint_B  \abs{V(\nabla
       z) -V(\nabla h)}^2 \omega_B^p\,dx \eqsim 
       \dashint_B \phi_{\abs{\nabla z}} \big(\abs{\nabla z - \nabla  
       h}\big) \omega_B^p \,dx.
  \end{align}
    Let us estimate~$\textrm{I}_1$. Using Lemma~\ref{lem:AB-A} we get 
  \begin{align*}
    \textrm{I}_1
    &=
      \dashint_B \big( \mathcal{A}_B(\nabla z) -\mathcal{A}(x,\nabla z)\big) \cdot (\nabla z-\nabla
      h )\,dx 
    \\
    &\lesssim \dashint_B \abs{\mathcal{A}_B(\nabla
      z)-\mathcal{A}(x,\nabla z)} \cdot \abs{\nabla z-\nabla 
      h}\,dx
    \\
    &\lesssim \dashint_B \frac{\abs{\bbM_B-\bbM}}{\abs{\bbM_B}+\abs{\bbM}}
      \big(\abs{\mathcal{A}_B(\nabla z)}+ \abs{\mathcal A(x,\nabla
      z)}\big)\abs{\nabla z-\nabla h} \, dx
    \\
    &\lesssim \dashint_B \frac{\abs{\bbM_B-\bbM}}{\abs{\bbM_B}+\abs{\bbM}}
      \big(\omega_B^p \phi'(\abs{\nabla z})+ \omega^p
      \phi'(\abs{\nabla z}) \big)\abs{\nabla z-\nabla h} \, dx.
      \end{align*} 
  Now we use Young's inequality and Lemma \ref{lem:hammer}
  \begin{align*}
    \textrm{I}_1&\le \sigma \dashint_B \phi_{\abs{\nabla z}}\big (
                  \abs{\nabla z -\nabla h}\big) \omega_B^p\,
                  dx
    \\
                &\quad +c_\sigma \dashint_B (\phi_{\abs{\nabla z}})^*
                  \bigg(\frac{\abs{\bbM_B-\bbM}}{\abs{\bbM_B}+\abs{\bbM}}
                  \phi'(\abs{\nabla z}) \bigg) \omega_B^p \,
                  dx
    \\
                &\quad +c_\sigma \dashint_B (\phi_{\abs{\nabla z}})^*
                  \bigg(\frac{\abs{\bbM_B-\bbM}}{\abs{\bbM_B}+\abs{\bbM}}
                  \frac{\omega^p}{\omega_B^p} 
                  \phi'(\abs{\nabla z}) \bigg) \omega_B^p \,
                  dx
                  :=\textrm{I}_{1,1}+\textrm{I}_{1,2}+\textrm{I}_{1,3},
  \end{align*}
  Now, using $(\phi_a)^*(\lambda t) \leq c\, (\lambda^2 +\lambda^{p'})
  \phi^{*}_a(t)$ for $a,\lambda,t \geq 0$ we obtain
  \begin{align*}
    \textrm{I}_{1,2} +
    \textrm{I}_{1,3}
    \leq c_\sigma \dashint_B (\phi_{\abs{\nabla z}})^*
    \bigg(\frac{\abs{\bbM_B-\bbM}}{\abs{\bbM_B}+\abs{\bbM}}
    \phi'(\abs{\nabla z}) \bigg) \bigg( 1 +
    \frac{\omega^{2p}}{\omega_B^{2p}} +
    \frac{\omega^{pp'}}{\omega_B^{pp'}}
    \bigg) \omega_B^p \,
    dx
  \end{align*}
  Now,~\eqref{eq:phialambdaa} and $(\phi_{\abs{a}})^*(\phi'(\abs{a})) \eqsim
  \phi(\abs{a})$ imply
  \begin{align*}
    \textrm{I}_{1,2} +
    \textrm{I}_{1,3}
    &\leq c_\sigma \dashint_B  \bigg( \frac{\abs{\bbM_B-\bbM}}{
    \abs{\bbM_B}+\abs{\bbM}} \bigg)^2
    (\phi_\abs{\nabla z})^*\big(\phi'(\abs{\nabla z})\big)
    \bigg( 1 +
    \frac{\omega^{2p}}{\omega_B^{2p}} +
    \frac{\omega^{pp'}}{\omega_B^{pp'}}
    \bigg) \omega_B^p \,
    dx
    \\
    &\eqsim c_\sigma \dashint_B  \bigg( \frac{\abs{\bbM_B-\bbM}}{
      \abs{\bbM_B}+\abs{\bbM}} \bigg)^2
      \phi(\abs{\nabla z})
      \bigg( \frac{\omega_B^p}{\omega^p} +
      \frac{\omega^{p}}{\omega_B^{p}} +
      \frac{\omega^{p'}}{\omega_B^{p'}}
      \bigg) \omega^p \,
      dx
  \end{align*}
  Now, we can use H\"older's inequality to conclude with
  Proposition~\ref{pro:small-scalar} 
  \begin{align*}
    \textrm{I}_{1,2}+\textrm{I}_{1,3}
    &\lesssim c \bigg( \dashint_B \bigg(
      \frac{\abs{\bbM_B-\bbM}}{\abs{\bbM_B}+\abs{\bbM}}
      \bigg)^{4s'}\, dx \bigg)^{\frac 1 {2s'} } \bigg( \dashint_B
      (\abs{\nabla z}^p \omega^p)^s \, dx\bigg    )^{\frac 1 {s}} 
    \\
    &\quad \quad \cdot \Bigg(
       \dashint_B 
      \bigg( \frac{\omega_B^p}{\omega^p} +
      \frac{\omega^{p}}{\omega_B^{p}} +
      \frac{\omega^{p'}}{\omega_B^{p'}}
            \bigg)^{2s'}\,
      dx \Bigg)^{\frac 1 {2s'}}
    \\
    &\lesssim c (2^p + 2^{p'})\bigg( \dashint_B \bigg(
      \frac{\abs{\bbM_B-\bbM}}{\abs{\bbM_B}+\abs{\bbM}}
      \bigg)^{4s'}\, dx \bigg)^{\frac 1 {2s'} } \bigg( \dashint_B
      (\abs{\nabla z}^p \omega^p)^s \, dx\bigg    )^{\frac 1 {s}} .
  \end{align*}
  Now, Proposition~\ref{pro:small} and the additional smallness
  assumption $\abs{\log \bbM}_{\setBMO(B)} \leq \kappa_1$ implies
  \begin{align*}
    \textrm{I}_{1,2}+\textrm{I}_{1,3}
    &\lesssim c\, (2 s')^2 \abs{\log \bbM}_{\setBMO(B)}^2 \bigg( \dashint_B
      (\abs{\nabla z}^p \omega^p)^s \, dx\bigg    )^{\frac 1 {s}} .
  \end{align*}
  Now we estimate $\textrm{I}_2$ as
  \begin{align*}
    \textrm{I}_2
    &=      \dashint_B \big( \mathcal{A}(x,\nabla
      z)-\mathcal{A}(x,\nabla z+g) \big) \cdot (\nabla z-\nabla
      h )\,dx
    \\
    &\le c \dashint_{B_r}\phi'_{\abs{\nabla z}}(\abs{g})\,\abs{\nabla
      z-\nabla h} \omega^p\, dx.
  \end{align*}
  By Young's inequality with $\phi_{\abs{\nabla z}}$ we get for some~$\sigma>0$
  \begin{align*}
    \textrm{I}_2
    &\le \sigma \dashint_{B_r}\phi_{\abs{\nabla z}}(\abs{\nabla
      z-\nabla h}) \omega_B^p dx
      +  c_\sigma \dashint_{B_r}(\phi_{\abs{\nabla
      z}})^*\bigg( \phi_{\abs{\nabla
      z}}' ( \abs{g} ) \frac{\omega^p}{\omega_B^p}\bigg)
      \omega_B^p\, dx
    \\
    &=: \textrm{I}_{2,1} + \textrm{I}_{2,2}.
  \end{align*}
  We fix~$\sigma>0$ so small such that
  $\textrm{I}_{2,1} \leq \frac 18 \textrm{I}_0$. Now, using that $\sigma$ is
  fixed, we can replace~$c_\sigma$ in $\textrm{I}_{2,2}$ by~$c$.  We calculate
  \begin{align*}
    \textrm{I}_{2,2}
    &\leq c\,\dashint_{B_r} (\phi_{\abs{\nabla
      z}})^*\Big( \phi_{\abs{\nabla
      z}}' ( \abs{g} )\Big) \bigg(
      \frac{\omega^{pp'}}{\omega_B^{pp'}} +
      \frac{\omega^{2p}}{\omega_B^{2p}} \bigg)
      \omega_B^p\, dx
    \\
    &\leq c\,\dashint_{B_r} \phi_{\abs{\nabla
      z}}( \abs{g} ) \bigg(
      \frac{\omega^{pp'}}{\omega_B^{pp'}} +
      \frac{\omega^{2p}}{\omega_B^{2p}} \bigg)
      \omega_B^p\, dx.
  \end{align*}
  With Lemma~\ref{lem:removal-shift} we can remove the shift
  from~$\phi_{\abs{\nabla z}}$ and obtain
  \begin{align*}
    \textrm{I}_{2,2} 
    &\leq \delta\,c
      \dashint_{B_r}\abs{\nabla z}^p \bigg(
      \frac{\omega^{pp'}}{\omega_B^{pp'}} +
      \frac{\omega^{2p}}{\omega_B^{2p}} \bigg)
      \frac{\omega_B^p}{\omega^p} \omega^p \,dx + c\, \delta^{1-p}
      \dashint_{B_r}\abs{g}^p \bigg(
      \frac{\omega^{pp'}}{\omega_B^{pp'}} +
      \frac{\omega^{2p}}{\omega_B^{2p}} \bigg)       \frac{\omega_B^p}{\omega^p} \omega^p\, dx.
  \end{align*}
  As before we can use H\"older's inequality and
  Proposition~\ref{pro:small-scalar} to get rid of the extra weight
  factors at the expense of a slightly larger power.
  \begin{align*}
    \textrm{I}_{2,2} 
    &\leq \delta\,c
      \bigg( \dashint_{B_r}\big(\abs{\nabla z}^p  \omega^p\big)^s \,dx
      \bigg)^{\frac 1s}+ c\, \delta^{1-p} \bigg(
      \dashint_{B_r}\big(\abs{g}^p \omega^p\big)^s
      \, dx \bigg)^{\frac 1s}.
  \end{align*}
  Using $g=(u - \mean{u}_{2B_0})
      p' \zeta^{p'-1} \nabla \zeta$ we get
  \begin{align*}
    \textrm{I}_{2,2} 
    &\leq \delta\,c
      \bigg( \dashint_{B_r}\big(\abs{\nabla z}^p  \omega^p\big)^s \,dx
      \bigg)^{\frac 1s}+ c\, \delta^{1-p} \bigg(
      \dashint_{B_r}\bigg(\frac{\abs{u - \mean{u}_{2B_0}}^p}{R^p}
      \omega^p\bigg)^s 
      \, dx \bigg)^{\frac 1s}.
  \end{align*}
  Let us estimate~$\textrm{I}_3$. We estimate with Young's inequality
  and $0 \leq \zeta \leq 1$
  \begin{align*}
    \textrm{I}_3 &\lesssim \dashint_B \zeta^{p-1} \omega^p \abs{G}^{p-1}
                   \bigg(\zeta \abs{\nabla z
                   - \nabla h} +  \frac{\abs{z-h}}{r} \bigg)\,dx
    \\
                 &\leq \delta^{1-p}\,c \dashint_B \frac{\omega^{p'}}{\omega_B^{p'}}
                   \zeta^p \abs{G}^p \omega^p\,dx
                   + \delta\,c\dashint_B \abs{\nabla z - \nabla h}^p
                   \omega_B^p\,dx
                   + \delta\,c\dashint_B \biggabs{\frac{z -
                   h}{r}}^p
                   \omega_B^p\,dx
    \\
    &=: \textrm{I}_{3,1}+ \textrm{I}_{3,2} + \textrm{I}_{3,3}.
  \end{align*}
  For the calculations that follow consider the term~$z-h$ to be
  extended outside of~$B$ by zero. By the weighted \Poincare{} estimate
  of Proposition~\ref{pro:DreDur} we estimate $\textrm{I}_{3,3}
  \lesssim \textrm{I}_{3.2}$.
  By triangle inequality and the minimizing property of~$h$,
  see~\eqref{eq:min-h}, we obtain
  \begin{align*}
    \textrm{I}_{3,2} &\leq \delta\, c \bigg( \dashint_B   \abs{\nabla z}^p
                       \omega_B^p\,dx + \dashint_B   \abs{\nabla h}^p
                       \omega_B^p\,dx\bigg) \leq \delta\, c\, \dashint_B   \abs{\nabla z}^p
                       \omega_B^p\,dx.
  \end{align*}
  As before we can use H\"older's inequality and
  Proposition~\ref{pro:small-scalar} to correct the weight slightly at
  the expense of a slightly larger power. We get
  \begin{align*}
    \textrm{I}_{3,2} &\leq  \delta\, c\, \bigg(\dashint_B   \big(\abs{\nabla z}^p
                       \omega^p\big)^s\,dx \bigg)^{\frac 1s}.
  \end{align*}
  By the same trick
  \begin{align*}
    \textrm{I}_{3,1} &\leq \delta^{1-p}\,c \bigg(\dashint_B 
                   \big(\zeta^p \abs{G}^p \omega^p\big)^s\,dx \bigg)^{\frac
                       1s}. 
  \end{align*}
  It remains to estimate~$\textrm{I}_4$.
  For the calculations that follow consider the term~$z-h$ to be
  extended outside of~$B$ by zero. From here we need to distinguish cases~$p>2$ and~$1<p<2$.   
  
  We start from  the case~$p>2$.  Using  $\nabla z+g=\zeta^{p'} \nabla u$  we get 
  \begin{align*}
    \textrm{I}_{4}&\leq \dashint_B
                    \abs{\nabla (\zeta^{p})}\abs{\mathcal{A}(x,\nabla
                    u)}\cdot\abs{(z-h)}\, dx
    \\
                  &\lesssim \dashint_B \zeta^{\frac 1 {p-1}}
                    \abs{\nabla\zeta}\abs{\zeta^{p'}\nabla u}^{p-2}
                    \omega^{p-1}\abs{\nabla u}\cdot \abs{z-h}
                    \omega\, dx
    \\
                  &\lesssim \dashint_B  \frac r
                    R\abs{\nabla
                    z+g}^{(p-2)}\abs{\nabla u} \biggabs{\frac
                    {z-h} r}\, \omega^p\, dx
    \\
                  &\lesssim \delta\bigg( \dashint_B \abs{\nabla z
                    +g}^p\omega^p\, dx\bigg) +c_\delta\bigg(
                    \dashint_B \frac{r^p}{R^p}\abs{\nabla
                    u}^p\omega^p\, dx\bigg) +\delta\bigg(\dashint_B
                    \biggabs{\frac{z-h}r}^p\omega^p\, dx\bigg)
    \\
                  &:= \textrm{I}_{4,1}+\textrm{I}_{4,2}+\textrm{I}_{4,3},
  \end{align*}
  where we have applied Young's inequality with
  exponents~$\frac{p}{p-2}$,~$p$,~$p$ at the last step. We estimate
  now using triangle inequality and
  $g=-(u - \mean{u}_{2B_0}) p' \zeta^{p'-1} \nabla \zeta$,
  see~\eqref{eq:def-g},
  \begin{align*}
    \textrm{I}_{4,1}
    &\lesssim \delta\bigg( \dashint_B \abs{g}^p\omega^p\,
      dx\bigg)+\delta \bigg( \dashint_B \abs{\nabla z}^p\omega^p\,
      dx\bigg)
    \\
    &
      \lesssim \delta\bigg( \dashint_B \biggabs{\frac{u -\mean{u}_{2B_0}}{R}}^p\omega^p\, dx\bigg)+\delta \bigg( \dashint_B \abs{\nabla z}^p\omega^p\, dx\bigg).
  \end{align*}
  % Later we use the definition of the funcion~$g$, see~\eqref{eq:def-g}, to estimate  the $\textrm{I}_{4,1}$ as the lower order term.  
  Using Caccioppoli inequality  we get
  \begin{align*}
    \textrm{I}_{4,2}
    &\lesssim  c_\delta \frac{r^p}{R^p} \bigg(
      \dashint_{B} \abs{\nabla u}^p\omega^p\,
      dx\bigg)
    \\
    &\lesssim  c_\delta  \frac{r^p}{R^p} \bigg(
      \dashint_{2B}\bigg(\frac{\abs{u-\mean{u}_{2B}}}{r}\bigg)^p\omega^p\,dx+
      \dashint_{2B} \abs{G}^p\omega^p\, dx\bigg)
    \\
    &\lesssim c_\delta\dashint_{2B} \biggabs{\frac{u-\mean{u}_{2B_0}}{R}}^p\omega^p\,dx+  c_\delta  \dashint_{2B} \abs{G}^p\omega^p\, dx ,
  \end{align*}
  where we used that $r<R$ and changed the mean value $\mean{u}_{2B}$
  to the worse approximation $\mean{u}_{2B_0}$.  It remains to
  estimate~$\textrm{I}_{4,3}$. Using \Poincare{}'s inequality of
  Proposition~\ref{pro:DreDur} we get
  triangle inequality, and the minimizing property of~$h$ we get for
  some $\theta \in (0,1)$
  \begin{align*}
    \textrm{I}_{4,3}
    &\lesssim \delta\dashint_B
      \biggabs{\frac{z-h}{r}}^p\omega^p\,dx
      \lesssim \delta\bigg(\dashint_B\big(\abs{\nabla z-\nabla
      h}\omega\big)^{\theta p}\,dx \bigg)^{\frac 1 \theta}
  \end{align*}
  As before we can correct the weight~$\omega^p$ to~$\omega_B^p$  by
  H\"older's inequality and Proposition~\ref{pro:small-scalar} and
  then use the triangle inequality and the minimizing property of~$h$.
  \begin{align*}
    \textrm{I}_{4,3}
    &\lesssim\delta\dashint_B\abs{\nabla z-\nabla h}^p\omega_B^p\,dx
    \lesssim\delta\dashint_B\abs{\nabla z}^p\omega_B^p\,dx.
  \end{align*}
  Now, we can change the weight~$\omega_B^p$ back to~$\omega^p$ by
  H\"older's inequality and Proposition~\ref{pro:small-scalar} at the   expense of a slightly increased exponent~$s>1$.
  \begin{align*}
    \textrm{I}_{4,3}
    &\lesssim\delta\dashint_B\abs{\nabla z-\nabla h}^p\omega_B^p\,dx
    \lesssim\delta\bigg(\dashint_B\big(\abs{\nabla z}^p\omega^p\big)^s\,dx
      \bigg)^{\frac 1s}.
  \end{align*}
  This completes the case~$p>2$.
  
  For case~$1<p<2$ we estimate with Young's inequality
  \begin{align*}
    \textrm{I}_4 
    &\lesssim \dashint  \abs{\nabla \zeta} \zeta^{p-1} \abs{\nabla
      u}^{p-1} \abs{z-h} \, \omega^p\,dx
    \\
                 &\leq   c\, \delta^{1-p}
                   \dashint_B  \abs{\nabla u}^p  \frac{r^{p'}}{R^{p'}}
                   \frac{\omega^{pp'}}{\omega_B^{p'}}\,dx 
                   + \delta\,\dashint_{B}  \biggabs{\frac{z-h}{r}}^p \omega_B^p\,dx 
    \\
                 &\leq c\, \delta^{1-p}
                   \dashint_B  \abs{\nabla u}^p \frac{r^{p'}}{R^{p'}} \frac{\omega^{pp'}}{\omega_B^{p'}}\,dx
                   + \delta\,\dashint_B   \abs{\nabla z - \nabla h}^p
                   \omega_B^p\,dx
    \\
                 &=: \textrm{I}_{4,1} + \textrm{I}_{4,2}.
  \end{align*}
  By triangle inequality and the minimizing property of~$h$,
  see~\eqref{eq:min-h}, we obtain
  \begin{align*}
    \textrm{I}_{4,2} &\leq \delta\, c \bigg( \dashint_B   \abs{\nabla z}^p
      \omega_B^p\,dx + \dashint_B   \abs{\nabla h}^p
      \omega_B^p\,dx\bigg) \leq \delta\, c\, \dashint_B   \abs{\nabla z}^p
                       \omega_B^p\,dx
  \end{align*}
  As before, we can correct the weight $\omega_B^p$ to $\omega^p$ by
  H\"older's inequality and Proposition~\ref{pro:small-scalar} at the
  expense of a slightly increased exponent~$s>1$. We get
  \begin{align*}
    \textrm{I}_{4,2} &\leq \delta\, c \bigg( \dashint_B \big( \abs{\nabla z}^p
                       \omega^p\big)^s\,dx\bigg)^{\frac 1s}.
  \end{align*}
  Also at the term~$\textrm{I}_{4,1}$ we can correct the weight
  $\omega^{pp'} \omega_B^{-{p'}}$ to~$\omega^p$ and obtain
  \begin{align*}
    \textrm{I}_{4,1}
    &\leq c\, \delta^{1-p} \frac{r^{p'}}{R^{p'}}
      \bigg( \dashint_B  \big(\abs{\nabla u}^p \omega^p\big)^s
      \,dx \bigg)^{\frac 1s}.
  \end{align*}
  Now, by Corollary~\ref{cor:small-high-int} and the Caccioppoli inequality from
  Proposition~\ref{pro:caccioppoli} and $r \leq R$ we get
  \begin{align*}
    \textrm{I}_{4,1}
    &\leq 
      c\,  \delta^{1-p} \frac{r^{p'}}{R^{p'}} \dashint_{2B}\abs{\nabla u}^p
      \omega^p \,dx + c \,  \delta^{1-p} \frac{r^{p'}}{R^{p'}}
      \bigg( \dashint_{2B} \big(\abs{G}^{p} \omega^p\big)^s\,dx
      \bigg)^{\frac 1s}
    \\
    &\leq 
      c\,  \delta^{1-p} \frac{r^{p}}{R^{p}} \dashint_{2B}\abs{\nabla u}^p
      \omega^p \,dx + c \,  \delta^{1-p} \frac{r^{p'}}{R^{p'}}
      \bigg( \dashint_{2B} \big(\abs{G}^{p} \omega^p\big)^s\,dx
      \bigg)^{\frac 1s}
    \\
    &\leq 
      c\,  \delta^{1-p} \dashint_{4B}\frac{\abs{u -
      \mean{u}_{4B}}^p}{R^p} \omega^p \,dx + c \, \delta^{1-p}\frac{r^{p'}}{R^{p'}}
      \bigg( \dashint_{4B} \big(\abs{G}^{p} \omega^p\big)^s\,dx
      \bigg)^{\frac 1s}
    \\
    &\leq 
      c\, \delta^{1-p}
      \dashint_{4B}\frac{\abs{u-\mean{u}_{2B_0}}^p}{R^p} \omega^p \,dx
      + c \,  \delta^{1-p}
      \bigg( \dashint_{4B} \big(\abs{G}^{p} \omega^p\big)^s\,dx
      \bigg)^{\frac 1s},
  \end{align*}
  where, in the second line, we used that, since $1<p<2$, we have $p'>p$.

  Collecting all estimates proves the proposition.
\end{proof}

\subsection{Decay Estimates}
\label{ssec:decay-estimates}

We will now use the comparison estimate to derive certain decay
estimates of~$\mathcal{V}(\cdot,\nabla u)$.
\begin{proposition}[Decay estimate]
  \label{pro:decay-estimate}
  Recall, that $B=B_r$, $B_0=B_R(x_0)$ and let $4B \subset 2B_0$.
  There exist $\lambda \in (0, \frac 12)$, $s>1$ and
  $\kappa_5=\kappa_5(p,n,\Lambda,s)$ such that the following holds: If
  $\abs{\log \bbM }_{\setBMO(2B)}\le \kappa_5$, then for
  every~$\delta \in (0,1)$ there holds
  \begin{align*}
    \lefteqn{\dashint_{\lambda B} \abs{\mathcal{V}(x,\nabla z)  -
    \mean{\mathcal{V}(x,\nabla z)}_{\lambda B}}^2\,dx } \quad
    &
    \\
    &\leq
      \tfrac 14 \dashint_{B}  \abs{\mathcal{V}(x,\nabla z)  -
      \mean{\mathcal{V}(x,\nabla z)}_{B}}^2\,dx
    \\
    &\quad +  c\, \big(\abs{\log \bbM}_{\setBMO(B)}^2 +
      \delta \big) \dashint_B
      \abs{\mathcal{V}(x,\nabla z)}^2\,dx
    \\
    &\quad  +c\lambda^{-n} \, \delta^{1-p} \Bigg(
      \dashint_{4B} \!\!\bigg(\frac{\abs{u-\mean{u}_{2B_0}}^p}{R^p}
      \omega^p \bigg)^s \,dx \Bigg)^{\frac 1s} + c\,
      \delta^{1-p} \bigg(\dashint_{4 B}
      \big(\zeta^p \abs{\mathcal{V}(x, G)}^2 \big)^s\,dx \bigg)^{\frac 1s}.
  \end{align*}
\end{proposition}
\begin{proof}
  Let $\abs{\log \bbM}_{\setBMO(2B)} \leq \kappa_4$ with $\kappa_4$ as
  in Proposition~\ref{pro:comparison}.  Also let~$s>1$ be as in
  Proposition~\ref{pro:comparison}. Let
  $\lambda \in (0,\frac 12)$, whose   precise value will be chosen
  later. Then
  \begin{align*}
    \textrm{I}_1     &:= \lefteqn{\dashint_{\lambda B} \abs{\mathcal V(x,\nabla z)  -
                       \mean{\mathcal V(x,\nabla z)}_{\lambda B}}^2\,dx  } \quad
    &
    \\
                     &\leq  c\,
                       \dashint_{\lambda B} \abs{\mathcal V(x,\nabla h)  -
                       \mean{\mathcal V(x,\nabla h)}_{\lambda B}}^2\,dx 
                       +  c\,\dashint_{\lambda B} \abs{\mathcal V(x,\nabla z)  -
                       \mathcal V(x,\nabla h)}^2\,dx
    \\
                     &=: \textrm{I}_2 + \textrm{I}_3.
  \end{align*}
  Now, 
  \begin{align*}
    \textrm{I}_3
    &\leq c\, \lambda^{-n}   \dashint_{\frac 12 B}
      \abs{\mathcal V(x,\nabla z)  - 
      \mathcal V(x,\nabla h)}^2\,dx
    \\
    &\leq  c\, \lambda^{-n}   \dashint_{\frac 12 B}
      \abs{\mathcal{V}_B(\nabla z)  - 
      \mathcal{V}_B(\nabla h)}^2\,dx 
    \\
    &\quad
      +   c\, \lambda^{-n}   \dashint_{\frac 12 B}
      \abs{\mathcal V(x,\nabla z)  - 
      \mathcal{V}_B(\nabla z)}^2\,dx
      +   c\, \lambda^{-n}   \dashint_{\frac 12 B}
      \abs{\mathcal V(x,\nabla h)  - 
      \mathcal{V}_B(\nabla h)}^2\,dx
    \\
    &=: \textrm{I}_{3,1} + \textrm{I}_{3,2} + \textrm{I}_{3,3}.
  \end{align*}
  With the comparison of Proposition~\ref{pro:comparison} we get
  \begin{align*}
    \textrm{I}_{3,1}
    &\leq c\, \lambda^{-n}    \dashint_{B}
      \abs{\mathcal{V}_B(\nabla z)  - 
      \mathcal{V}_B(\nabla h)}^2\,dx
    \\
    &\leq c\, \lambda^{-n} \, \big(\abs{\log
      \bbM}_{\setBMO(B)}^2 + \delta \big) 
      \bigg( \dashint_B (\abs{\nabla z}^p \omega^p)^s \, dx\bigg
      )^{\frac 1 {s}} 
    \\
    &\quad +     c\lambda^{-n} \, \delta^{1-p} \Bigg(
      \dashint_{4B} \bigg(\frac{\abs{u-\mean{u}_{2B_0}}^p}{R^p}
      \omega^p \bigg)^s \!dx \Bigg)^{\frac 1s} \!\!\!
      + c \lambda^{-n} \,  \delta^{1-p}
      \bigg( \dashint_{4B} \big(\zeta^p \abs{G}^{p} \omega^p\big)^s\,dx
      \bigg)^{\frac 1s}\!\!.
  \end{align*}
  For the terms~$\textrm{I}_{3,2}$ and $\textrm{I}_{3,3}$ we will
  proceed similar to the proof of Lemma~\ref{lem:AB-A}.  Note that
  $\abs{\bbM_B \xi} \eqsim \abs{\bbM_B}\abs{\xi}$ and
  $\abs{\bbM \xi} \eqsim \abs{\bbM} \abs{\xi}$ due to
  \eqref{eq:mon-Mxi}, \eqref{eq:mon-MBxi}
  and~\eqref{eq:mon-MBest}. So Lemma~\ref{lem:hammer} and \eqref{eq:phialambdaa} imply
  \begin{align*}
    \textrm{I}_{3,2}
    &\leq c\, \lambda^{-n} \dashint_B \abs{V(\bbM \nabla z) - V(\bbM_B
      \nabla z)}^2\,dx
    \\
    &\leq c\, \lambda^{-n} \dashint_B \phi_{\abs{\bbM \nabla z} \vee
      \abs{\bbM_B \nabla z}}(\abs{\bbM(x) \nabla z - \bbM_B \nabla z})\,dx
    \\
    &\leq c\, \lambda^{-n} \dashint_B \bigg(\frac{\abs{\bbM \nabla z - \bbM_B
      \nabla z}}{\abs{\bbM \nabla z} \vee 
      \abs{\bbM_B \nabla z}} \bigg)^2
      \Big(\phi(\abs{\bbM(x) \nabla z}) +  \phi(\abs{\bbM_B
      \nabla z})\Big) \,dx
    \\
    &\leq c\, \lambda^{-n}       \dashint_B \bigg(\frac{\abs{\bbM - \bbM_B}
      }{\abs{\bbM} \vee 
      \abs{\bbM_B}} \bigg)^2 \Big(
      \phi(\abs{\bbM(x) \nabla z}) +  \phi(\abs{\bbM_B \nabla z}) \Big)\,dx
    \\
    &\leq c\, \lambda^{-n}       \dashint_B \bigg(\frac{\abs{\bbM - \bbM_B}
      }{\abs{\bbM} \vee 
      \abs{\bbM_B}} \bigg)^2 \Big(
      \abs{\nabla z}^p \omega^p + \abs{\nabla z}^p \omega_B^p \Big)\,dx
    \\
    &\leq c\, \lambda^{-n}       \dashint_B \bigg(\frac{\abs{\bbM - \bbM_B}
      }{\abs{\bbM} \vee 
      \abs{\bbM_B}} \bigg)^2 \bigg( 1 + \frac{\omega_B^p}{\omega^p}
      \bigg) 
      \abs{\nabla z}^p \omega^p\,dx
  \end{align*}
  Let $s>1$ be as in Corollary~\ref{cor:small-high-int}. Then
  H\"older's inequality with exponents $(2s',2s',s)$ implies
  \begin{align*}
    \textrm{I}_{3,2}
    &\leq c\, \lambda^{-n} \bigg( \dashint_B \bigg(
      \frac{\abs{\bbM_B-\bbM}}{\abs{\bbM_B}+\abs{\bbM}}
      \bigg)^{4s'}\, dx \bigg)^{\frac 1 {2s'} } \bigg( \dashint_B
      (\abs{\nabla z}^p \omega^p)^s \, dx\bigg    )^{\frac 1 {s}} 
    \\
    &\quad \quad \cdot \Bigg( 1 +
      \bigg( \dashint_B \Big(\frac{\omega_B^p}{\omega^p}\Big)^{2s'}\,
      dx\bigg)^{\frac 1 {2s'}} \Bigg).
  \end{align*}
  With Propositions~\ref{pro:small} and~\ref{pro:small-scalar} and
  Corollary~\ref{cor:small-high-int} we obtain 
  \begin{align*}
    \textrm{I}_{3,2}
    &\leq c\, \, \abs{\log \bbM}_{\setBMO(B)}^2 \lambda^{-n} \bigg( \dashint_B
      (\abs{\nabla z}^p \omega^p)^s \, dx\bigg    )^{\frac 1 {s}}
    \\
    &\leq c\lambda^{-n}  \abs{\log \bbM}_{\setBMO(B)}^2
      \Bigg(\dashint_{2B} \abs{\mathcal V(x,\nabla z)}^2\, dx
      +\bigg(\dashint_{2B} \abs{\mathcal V(x, G))} ^{2s}  \bigg)^\frac
      {1}{s} \Bigg) 
  \end{align*}
  Analogously, as with~$\textrm{I}_{3,2}$ we estimate
  \begin{align*}
    \textrm{I}_{3,3}
    &\leq c\, \lambda^{-n} \dashint_{\frac 12 B} \abs{V(\bbM \nabla h) - V(\bbM_B
      \nabla h)}^2\,dx
    \\
    &\leq c\, \lambda^{-n}       \dashint_{\frac 12 B} \bigg(\frac{\abs{\bbM - \bbM_B}
      }{\abs{\bbM} \vee 
      \abs{\bbM_B}} \bigg)^2 \Big(
      \abs{\nabla h}^p \omega^p + \abs{\nabla h}^p \omega_B^p
      \Big)\,dx
    \\
    &\leq c\, \lambda^{-n} \big(\max_{\frac 12 B} \omega_B^p
      \abs{\nabla h}^p\big)      \dashint_{\frac 12 B}
      \bigg(\frac{\abs{\bbM - \bbM_B} 
      }{\abs{\bbM} \vee 
      \abs{\bbM_B}} \bigg)^2 \bigg( \frac{\omega^p}{\omega_B^p} + 1\bigg)\,dx.
  \end{align*}
  The interior regularity of~$h$, see Proposition~\ref{pro:reg-h}, and
  the minimizing property of~$h$ implies
  \begin{align*}
    \textrm{I}_{3,3}
    &\leq c\, \lambda^{-n} \dashint_B \abs{\nabla h}^p \omega_B^p\,dx      \dashint_{\frac 12 B}
      \bigg(\frac{\abs{\bbM - \bbM_B} 
      }{\abs{\bbM} \vee 
      \abs{\bbM_B}} \bigg)^2 \bigg( \frac{\omega^p}{\omega_B^p} +
      1\bigg)\,dx
    \\
    &\leq c\, \lambda^{-n} \dashint_B \abs{\nabla z}^p \omega_B^p\,dx      \dashint_{\frac 12 B}
      \bigg(\frac{\abs{\bbM - \bbM_B} 
      }{\abs{\bbM} \vee 
      \abs{\bbM_B}} \bigg)^2 \bigg( \frac{\omega^p}{\omega_B^p} +
      1\bigg)\,dx.
  \end{align*}
  With H\"older's inequality, Proposition~\ref{pro:small} and
  Proposition~\ref{pro:small-scalar} we obtain
  \begin{align*}
    \textrm{I}_{3,3}
    &\leq c\, \lambda^{-n} \dashint_B \abs{\nabla z}^p \omega_B^p\,dx
      \Bigg(\dashint_{B}
      \bigg(\frac{\abs{\bbM - \bbM_B} 
      }{\abs{\bbM} \vee 
      \abs{\bbM_B}} \bigg)^4 \,dx \Bigg)^{\frac 12}
      \Bigg( \bigg(\dashint_B
      \Big(\frac{\omega^p}{\omega_B^p}\Big)^2\,dx \bigg)^{\frac 12}
      + 1
      \bigg)
    \\
    &\leq c\, \abs{\log \bbM}_{\setBMO(B)}^2 \lambda^{-n} \dashint_B
      \abs{\nabla z}^p \omega_B^p\,dx.
  \end{align*}
  With H\"older's inequality, Proposition~\ref{pro:small-scalar} and
  Corollary~\ref{cor:small-high-int} we obtain
  \begin{align*}
    \textrm{I}_{3,3}
    &\leq c\, \abs{\log \bbM}_{\setBMO(B)}^2 \lambda^{-n} \bigg(\dashint_B
      \big(\abs{\nabla z}^p \omega^p\big)^s\,dx \bigg)^{\frac 1s}
      \Bigg( \dashint \bigg(\frac{\omega_B^p}{\omega^p} \bigg)^{s'}
      \,dx \Bigg)^{\frac 1{s'}}
    \\
    &\leq  c\, \abs{\log \bbM}_{\setBMO(B)}^2 \lambda^{-n} \bigg(\dashint_B
      \big(\abs{\nabla z}^p \omega^p\big)^s\,dx \bigg)^{\frac 1s}
    \\
    &\leq  c\, \abs{\log \bbM}_{\setBMO(B)}^2 \lambda^{-n} \Bigg( \dashint_B
      \abs{\mathcal V(x,\nabla z)}^2\,dx + \bigg(\dashint_B
      \abs{\mathcal V(x,G)}^{2s}\,dx \bigg)^{\frac 1s} \Bigg).
  \end{align*}
  The final estimate for the term~$\textrm{I}_3$ takes form 
  \begin{align*}
    \textrm{I}_3
    &\leq  c\, \big(\abs{\log \bbM}_{\setBMO(B)}^2 + \delta \big)\lambda^{-n}\dashint_B
      \abs{\mathcal V(x,\nabla z)}^2\,dx
    \\
    &\quad  + c\, \big( \abs{\log \bbM}_{\setBMO(B)}^2 +
      \delta^{1-p}\big) \lambda^{-n} \bigg(\dashint_B
      \abs{\mathcal V(x,G)}^{2s}\,dx \bigg)^{\frac 1s}
    \\
    &\quad +c\lambda^{-n} \, \delta^{1-p} \Bigg(
      \dashint_{4B} \bigg(\frac{\abs{u-\mean{u}_{2B_0}}^p}{R^p}
      \omega^p \bigg)^s \,dx \Bigg)^{\frac 1s}.
  \end{align*}
  We estimate now the term~$\textrm{I}_2$:
  \begin{align*}
    \textrm{I}_2
    &=  c\,
      \dashint_{\lambda B} \abs{\mathcal V(x,\nabla h)  -
      \mean{\mathcal V(x,\nabla h)}_{\lambda B}}^2\,dx
    \\
    &\leq  c\,\dashint_{\lambda B} \abs{\mathcal{V}_B(\nabla h)  -
      \mean{\mathcal{V}_B(\nabla h)}_{\lambda B}}^2\,dx + c\,\dashint_{\lambda B} \abs{\mathcal V(x,\nabla h)  -
      \mathcal{V}_B(\nabla h)}^2\,dx
    \\
    &:=\textrm{I}_{2,1}+\textrm{I}_{2,2}.
  \end{align*}
  We use the decay estimate from Proposition~\ref{pro:reg-h} to get
  \begin{align*}
    \textrm{I}_{2,1}
    &\le c\lambda^{2\alpha} \dashint_{\frac 12 B}  \abs{\mathcal{V}_B(\nabla h)  -
      \mean{\mathcal{V}_B(\nabla h)}_{ B}}^2\,dx.
  \end{align*}
  By several triangle inequalities we obtain
  \begin{align*}
    \textrm{I}_{2,1}
    &\le  c \lambda^{2\alpha} \dashint_{B}  \abs{\mathcal V(x,\nabla z)  -
      \mean{\mathcal V(x,\nabla z)}_{B}}^2\,dx
      + c \lambda^{2\alpha}  \dashint_{\frac 12 B}  \abs{\mathcal{V}_B(\nabla z)  -
    \mathcal{V}_B(\nabla h)}^2\,dx
    \\
    &\quad + c \lambda^{2\alpha}  \dashint_{\frac 12 B}  \abs{\mathcal V(x,\nabla z)  -
      \mathcal{V}_B(\nabla z)}^2\,dx
      +c \lambda^{2\alpha} \dashint_{\frac 12 B}  \abs{\mathcal V(x,\nabla h)  -
      \mathcal{V}_B(\nabla h)}^2\,dx
    \\
    &=: \textrm{I}_{2,1,0} + \textrm{I}_{2,1,1} + \textrm{I}_{2,1,2} + \textrm{I}_{2,1,3}.
  \end{align*}
  We can estimate $\textrm{I}_{2,1,1}$, $\textrm{I}_{2,1,2}$,
  $\textrm{I}_{2,1,3}$ and as $\textrm{I}_{3,1}$, $\textrm{I}_{3,2}$ and
  $\textrm{I}_{3,3}$, respectively, except that the
  factor~$\lambda^{-n}$ is replaced by~$\lambda^{2\alpha}$. Moreover, we have
  \begin{align*}
    \textrm{I}_{2,2} &\leq c\,\lambda^{-n} \dashint_{\frac 12 B} \abs{\mathcal V(x,\nabla h)  -
      \mathcal{V}_B(\nabla h)}^2\,dx
  \end{align*}
  which can be estimated as~$\textrm{I}_{3,3}$. Overall, we can
  estimate~$\textrm{I}_2$ exactly as~$\textrm{I}_3$ (with some better
  factors as some places) but get the additional
  term~$\textrm{I}_{2,1,0}$. We arrive at the final estimate
  \begin{align*}
    \textrm{I}_1 &\leq \textrm{I}_2 + \textrm{I}_3
    \\
                 &\leq
                   c \lambda^{2\alpha} \dashint_{B}  \abs{\mathcal V(x,\nabla z)  -
                   \mean{\mathcal V(x,\nabla z)}_{B}}^2\,dx
    \\
                 &\quad +  c\,  \delta \lambda^{-n}\dashint_B
                   \abs{\mathcal V(x,\nabla z)}^2\,dx
    \\
                 &\quad  + c\, \big( \abs{\log \bbM}_{\setBMO(B)}^2 +
                   \delta^{1-p}\big) \lambda^{-n} \bigg(\dashint_{4B}
                   \abs{\mathcal V(x,G)}^{2s}\,dx \bigg)^{\frac 1s}
    \\
                 & \quad +  c\lambda^{-n} \, \delta^{1-p} \Bigg(
                   \dashint_{4B} \bigg( \frac{\abs{u-\mean{u}_{2B_0}}^p}{R^p}
                   \omega^p \bigg)^s \,dx \Bigg)^{\frac 1s}.
  \end{align*}
  Now, we fix~$\lambda \in (0, \frac 12)$ such that the
  factor $c \lambda^{2\alpha}$ is smaller than~$\frac 14$. This proves
  the claim.
\end{proof}

%We can now summarize our results in certain estimates of maximal
%functions.  
For locally integrable function~$f$ we define the
Hardy-Littlewood maximal function and the sharp maximal function for
$\rho \in [1,\infty)$ by
\begin{align*}
  \mathcal M_\rho f
  (x)&:=\sup_{B(x)}\bigg(\dashint\limits_{B(x)}\!\!\!\abs{f}^\rho
       dy\!\bigg)^{\frac 1 \rho},  
  &
    \mathcal M^\sharp_{\rho}f(x)&:= \sup_{B(x)}\bigg(\dashint\limits_{B
                               (x)} \!\!\!|f-\mean{f}_{B(x)}|^\rho
                               dy\!\bigg)^{\frac 1 \rho}. 
\end{align*}
We can use these operators to express the decay estimates of
Proposition~\ref{pro:decay-estimate} in another form.
\begin{proposition}
  \label{pro:maximal-functions}
  There exists $s>1$ and $\kappa_5=\kappa_5(p,n,\Lambda,s)$ such that
  the following holds: If
  $\abs{\log \bbM }_{\setBMO(4B_0)}\le \kappa_5$, then for almost
  all~$x \in \Rn$
  \begin{align*}
    M^\sharp_2\big(\mathcal{V}(\cdot,\nabla z)\big)(x)
    &\leq  c\, \big(\abs{\log \bbM}_{\setBMO(2B_0)} +
      \delta \big) M_2\big(\mathcal{V}(\cdot,\nabla
      z)\big)(x)
    \\
    &\quad + c 
      \delta^{1-p} R^{-p} \Big(M_{2s} \big(\indicator_{4B_0} \abs{u-
      \mean{u}_{2B_0}}^p \omega^p\big)(x) \Big)^{\frac 12}
    \\
    &\quad  + c\, \delta^{1-p} \mathcal M_{2s}\big(\indicator_{B_0}
      \mathcal{V}(\cdot,G)\big)(x)
    \\
    &\quad + c\, \frac{R^n}{(R +\abs{x})^n}
      \bigg( \dashint_{B_0}\abs{\mathcal{V}(\cdot,\nabla z)-
      \mean{\mathcal{V}(\cdot,\nabla z)}_{B_0}}^2\,dx
      \bigg)^{\frac 12}.
  \end{align*}
\end{proposition}
\begin{proof}
  We choose~$\kappa_5$, $s$ and $\lambda \in (0,\frac 12)$ as in
  Proposition \ref{pro:decay-estimate}. Since
  $\mathcal{V}(\cdot, \nabla v) \in L^2(\Rn)$,
  $\mathcal{V}(\cdot,G) \in L^2(4B_0)$ and by
  Proposition~\ref{pro:DreDur}
  $\abs{u - \mean{u}_{2B_0}}^p \omega^p \in L^s(2B_0)$ all terms in
  the following calculations are finite at least for almost every~$x$.
  Fix $x\in \Rn$. Then
  \begin{align*}
    \textrm{I} :=
    M^\sharp_2\big(\mathcal{V}(\cdot,\nabla z)\big)(x)
    &= \sup_{r>0} \bigg(\dashint_{B_r(x)} \abs{\mathcal{V}(x,\nabla z)
      - \mean{\mathcal{V}(x,\nabla z)}_{B_r(x)}}^2\,dy \bigg)^{\frac 12}.
  \end{align*}
  We split the choice of $r\in(0,\infty)$ into three parts
  \begin{enumerate}
  \item $J_1 := \set{r>0:\,B_r(x_0) \cap B_0 = \emptyset}$.
  \item $J_2 := \set{r>0\,:\,\frac{2}{\lambda} B_r(x) \subset 4 B_0}$.
  \item
    $J_3 := \set{r>0\,:\, B_r(x_0) \cap B_0 = \emptyset \text{ and }
      \frac{2}{\lambda} B_r(x) \not\subset 4 B_0}$.
  \end{enumerate}
  For $k=1,2,3$ abbreviate
  \begin{align*}
    \textrm{I}_k &:= \sup_{r \in {J}_k}  \dashint_{B_r(x)} \abs{\mathcal{V}(x,\nabla z) - \mean{\mathcal{V}(x,\nabla z)}_{B_r(x)}}\,dy.
  \end{align*}
  Since $z=0$ outside of~$B_0$, we obviously have~$\textrm{I}_1=0$.
  If $r \in {J}_2$, then by the decay estimate of
  Proposition~\ref{pro:decay-estimate} applied to $B  = \lambda^{-1}
  B_r(x)$ (with $\delta$ replaced by~$\delta^2$) we get
  \begin{align*}
    \textrm{I}_2 &\leq \tfrac 14 \textrm{I}
                   + c\, \big(\abs{\log \bbM}_{\setBMO(B)} +
                   \delta \big) M_2\big(\mathcal{V}(\cdot,\nabla
                   z)\big)(x)
    \\
                 &\quad + c 
                   \delta^{1-p} R^{-p} \Big( \mathcal M_{2s} \big(\indicator_{4B_0} \abs{u-
                   \mean{u}_{2B_0}}^p \omega^p\big)(x) \Big)^{\frac 12}
                   + c\, \delta^{1-p} M_{2s}\big(\indicator_{B_0}
                   \mathcal{V}(\cdot,G)\big)(x).
  \end{align*}
  If $r \in J_3$, then $r \geq c\, R$. It follows with $\support z
  \subset \overline{B_0}$ that
  \begin{align*}
    \textrm{I}_3 
    &\leq c\, \frac{R^n}{(R +\abs{x})^n}
      \bigg( \dashint_{B_0}\abs{\mathcal{V}(\cdot,\nabla z)-
      \mean{\mathcal{V}(\cdot,\nabla z)}_{B_0}}^2\,dx
      \bigg)^{\frac 12}.
  \end{align*}
  Combining the estimate and absorbing~$\frac 14 \textrm{I}$ (which is
  finite for almost all~$x$) we prove the claim.
\end{proof}

\subsection{Main Result Non-Linear}
\label{ssec:main}

In this section we prove our main theorem~\ref{thm:main}. We will use
Proposition~\ref{pro:maximal-functions} to prove higher integrability
of~$\mathcal{V}(\cdot,\nabla z)$ and then as a consequence of
$\abs{\nabla u}^p \omega^p$. For this we need the famous
Fefferman-Stein inequality that allows to estimate the $L^q$-norm of
the maximal operator the $L^q$-norm of the sharp maximal operator,
i.e. for $q \in (2,\infty)$ there holds
\begin{align}
  \label{eq:fefferman-stein-aux}
  \norm{\mathcal M_2 f}_q\le c(q) \norm{\mathcal M^\sharp_2 f}_q
\end{align}
for all $f \in L^q$. This allows to absorb the term with
$\big(\abs{\log \bbM}_{\setBMO(B)} + \delta \big)
M_2\big(\mathcal{V}(\cdot,\nabla z)\big)$ later on the left-hand side.
This trick was already used in~\cite{KinZho99} and more recently in ~\cite{BCGOP17} in a slightly different
form. Kinnunen and Zhou used a local version of the Fefferman-Stein
inequality. Unfortunately, the constant $c(q)$ in the version
of~\cite[Lemma~2.4]{KinZho99} depends heavily on~$q$ and is not
adequate to obtain sharp estimates\footnote{There also exists other local
  version of the Fefferman-Stein estimate in \cite[Lemma~4]{Iwa83} or
  \cite[Theorem~5.25]{DieRuzSch10}.
  %  or \cite[Lemma~15]{BCGOP17}
  However, as far as we can see
  these versions
  % as well as the version in \cite{KinZho01}
  depend
  exponentially on~$q$.}. %
We therefore present a version of Fefferman-Stein inequality with
linear dependency on~$q$ (for $q$ large).
\begin{theorem}
  \label{thm:fefferman-stein}
  Let $q > 1$. Then
  \begin{align}
    \norm{f}_q &\le c\, q\,\norm{\mathcal M^\sharp_1 f}_q
  \end{align}
  for all $f \in L^q(\Rn)$.
\end{theorem}
\begin{proof}
  The proof\footnote{It is also possible to proof the theorem by
    redistributional estimates as in Chapter~IV, Section~3.6,
    Corollary~1 of~\cite{Stein93harmonic}. However, the dependency
    on~$q$ is again exponential.}  is based on the duality of the
  Hardy space~$\mathcal{H}^1$ and ~$\setBMO$, see Chapter~IV,
  Section~2 of~\cite{Stein93harmonic}. By the same truncation
  arguments as in \cite{Stein93harmonic} it suffices to
  consider~$f \in L^q(\Rn) \cap L^\infty(\Rn)$. Let
  $f \in \mathcal{H}^1(\Rn) \cap L^{q'}(\Rn)$, where $\mathcal{H}^1$
  is the Hardy space. Then by~(16) of~\cite{Stein93harmonic}
  \begin{align*}
    \skp{f}{g} &\leq c\,\skp{\mathcal{M}^\sharp f}{\mathcal{M} g}
                 \leq c\,\norm{\mathcal{M}^\sharp f}_q \norm{\mathcal{M} g}_{q'}.
  \end{align*}
  It is well known that
  \begin{align}
    \label{eq:bnd-M}
    \norm{\mathcal{M} g}_{q'} \leq c\, q\, \norm{g}_{q'},
  \end{align}
  see for example Chapter~I, Section~3, Theorem~1and
  Remark~\cite{Stein93harmonic}. Thus,
  \begin{align*}
    \skp{f}{g} &\leq c\,q\,\norm{\mathcal{M}^\sharp f}_q \norm{g}_{q'}.
  \end{align*}
  The claim follows by taking the supremum over
  all~$g \in L^{q'}(\Rn)$ with $\norm{g}_{p'} \leq 1$.
\end{proof}
To proceed, we  need the following lemma for improving reverse H{\"o}lder
estimates from \cite{DieKapSch12}. The lemma is a minor modification
of the \cite[Remark~6.12]{Giu03} and~\cite[Lemma~3.2]{DuzMin10}.
\begin{lemma}
  \label{lem:imprevH}
  Let $B \subset \Rn$ be a ball, let $g,h\,:\,\Omega \to
  \setR$ be a integrable functions and $\theta \in (0,1)$ such that
  \begin{align*}
    \dashint_B \abs{g}\,dx \leq c_0\, \bigg(\dashint_{2B}
    \abs{g}^{\theta} \,dx\bigg)^\frac{1}{\theta} + \dashint_{2B}
    \abs{h}\,dx.
  \end{align*}
  for all balls $B$ with $2B \subset \Omega$. Then for every $\gamma
  \in (0,1)$ there exists $c_1=c_1(c_0, \gamma)$ such that
  \begin{align*}
    \dashint_B \abs{g}\,dx \leq c_1\, \bigg( \dashint_{2B}
    \abs{g}^{\gamma} \,dx \bigg)^{\frac 1 \gamma} + c_1\,
    \dashint_{2B} \abs{h}\,dx.
  \end{align*}
\end{lemma}

We are now prepared to prove the estimate of our main result under the
assumption that the function~$u$ is already regular enough. We get rid
of this extra assumption later.
\begin{proposition}
  \label{pro:main-pre}
  Let $u$ be a local weak solution of~\eqref{eq:sysM}, let~$\bbM$
  satisfy~\eqref{eq:ass-M}, define~$\omega$ by~\eqref{eq:def-omega}.
  Then there exists $\kappa_6 = \kappa_6(p,n, \Lambda)$ such that for
  all balls~$B_0$ with $8B_0 \subset \Omega$ and all
  $\rho \in [p,\infty)$ with
  \begin{align}
    \label{eq:smallness-pre}
    \abs{\log \bbM}_{\setBMO(8B_0)} \leq \kappa_6 \frac{1}{\rho}
  \end{align}
  and $\abs{\nabla u} \omega \in L^\rho(B_0)$ there holds
  \begin{align*}
    \bigg(\dashint_{\frac 12 B_0} \big( \abs{\nabla u}\, \omega
    \big)^\rho\,dx\bigg)^{\frac 1 \rho} &\leq c_\rho \dashint_{4 B_0} 
                                   \abs{\nabla u}\, \omega \,dx + c_\rho \bigg( \dashint_{4 B_0}
                                   \big( \abs{G}\, \omega\big)^\rho\,dx \bigg)^{\frac 1 \rho}
  \end{align*}
  for all balls~$B_0$ with $8B_0 \subset \Omega$, where $c_\rho =
  c_\rho(p,n,\Lambda,\rho)$. The constant $c_\rho$ is continuous in~$\rho$.
\end{proposition}
\begin{proof}
  Define~$z$ as in the previous section and let~$\kappa_5$ as in
  Proposition~\ref{pro:maximal-functions}. We will
  choose~$\kappa_6 \leq \kappa_5 / p$.  Let $q := \rho/ p \geq 1$. If
  $1 \leq q \leq s$, then the claim already follows from
  Corollary~\ref{cor:small-high-int}. Thus, it suffices to consider
  the case~$q \geq s$. By replacing~$s$ by a smaller one in the steps
  above, we can even assume that $1 < s < s^2 < q$. The only reason
  for this assumption is to avoid exploding constants for~$q$ close
  to~$1$.

  It follows from Proposition~\ref{pro:maximal-functions} 
  \begin{align*}
    \textrm{I} :=
    \norm{\mathcal{M}_2^\sharp\mathcal{V}(\cdot,\nabla z)}_{2q}
    &\leq  c\, \big(\abs{\log \bbM}_{\setBMO(2B_0)} +
      \delta \big) \norm{M_2(\mathcal{V}(\cdot,\nabla
      z)}_{2q}
    \\
    &\quad + c\,
      \delta^{1-p} R^{-p} \bignorm{ \mathcal M_{2s} (\indicator_{4B_0} \abs{u-
      \mean{u}_{4B_0}}^p \omega^p\big)}_q^{\frac 12}
    \\
    &\quad  +c\, \delta^{1-p} \bignorm{\mathcal M_{2s}\big(\indicator_{4B_0}
      \mathcal{V}(\cdot,G)\big)}_2
    \\
    &\quad + c\,
      \biggnorm{\frac{R^n}{(R +\abs{x})^n}}_{2q}
      \bigg( \dashint_{B_0}\abs{\mathcal{V}(\cdot,\nabla z)-
      \mean{\mathcal{V}(\cdot,\nabla z)}_{B_0}}^2\,dx
      \bigg)^{\frac 12}
    \\
    &=: \textrm{I}_1 + \textrm{I}_2 + \textrm{I}_3 + \textrm{I}_4.
  \end{align*}
  Since $\abs{\nabla u}^p \omega^p \in L^q(B_0)$, we have
  $\mathcal{V}(\cdot,\nabla z)\in L^{2q}(\Rn)$. As a
  consequence~$\textrm{I} < \infty$.

  Now, by~\eqref{eq:bnd-M} (using
  $\mathcal{M}_2(g) = (\mathcal{M}(\abs{g}^2))^{\frac 12})$ and $s^2 <
  q$) and
  Theorem~\ref{thm:fefferman-stein} we obtain
  \begin{align*}
    \norm{\mathcal{M}_2(\mathcal{V}(\cdot,\nabla
    z)}_{2q} \leq c_s\, q\, 
    \norm{\mathcal{V}(\cdot,\nabla
    z)}_{2q} \leq c_s\, 
    \norm{\mathcal{M}_1^\sharp(\mathcal{V}(\cdot,\nabla
    z))}_{2q} \leq c_s\, 
    \norm{\mathcal{M}_2^\sharp(\mathcal{V}(\cdot,\nabla
    z))}_{2q}.
  \end{align*}
  We obtain
  \begin{align*}
    \textrm{I}_1 &\leq c\, q\,\big(\abs{\log \bbM}_{\setBMO(2B_0)} +
      \delta \big) \, \textrm{I}.
  \end{align*}
  Since $\abs{\nabla u}^p \omega^p \in L^q(B_0)$, we have
  $\abs{\mathcal{V}(\cdot,\nabla z)}^2\in L^q(\Rn)$.

  Now, we can fix~$\kappa_6$ and $\delta$ (choose $\delta \in O(1/q)$)
  so small such that
  \begin{align*}
    \textrm{I}_1 &\leq \tfrac 12 \textrm{I}.
  \end{align*}
  Thus, we can absorb~$\textrm{I}_1$ into~$\textrm{I}$.
  On the other hand we get
  \begin{align*}
    \textrm{I}_2 &\leq c\, q^{p-1}  \bigg( \int_{4 B_0} \bigg(\frac{ \abs{u -
                   \mean{u}_{4B_0}}^p}{R^p} \omega^p \bigg)^q \,dx
                   \bigg)^{\frac 1 {2q}}. 
  \end{align*}
  With $\abs{\mathcal{V}(\cdot, G)}^2 \leq c\, \abs{G}^p \omega^p$ we get
  and
  \begin{align*}
    \textrm{I}_3 &\leq c\, q^{p-1}
                   \bigg( \int_{4 B_0} \big( \abs{\mathcal{V}(\cdot,G)}^2 \omega^p\big)^q\,dx \bigg)^{\frac 1 {2q}}.
  \end{align*}
  Finally,
  \begin{align*}
    \textrm{I}_4 &\leq c\, \abs{B_0}^{\frac 1{2q}}
                         \bigg( \dashint_{B_0}\abs{\mathcal{V}(\cdot,\nabla z)}^2\,dx
      \bigg)^{\frac 12}
  \end{align*}
  We also use
  \begin{align*}
    \textrm{I} &=
                 \norm{\mathcal{M}_2^\sharp\mathcal{V}(\cdot,\nabla z)}_{2q}
                 \geq c\,    \norm{\mathcal{M}_1^\sharp\mathcal{V}(\cdot,\nabla z)}_{2q}
                 \geq \frac{c}{q}   \norm{\mathcal{V}(\cdot,\nabla z)}_{2q}.
  \end{align*}
  Overall, we obtain
  \begin{align*}
    \norm{\mathcal{V}(\cdot,\nabla z)}_{2q}
    &\leq
      c\, q^{p}  \bigg( \int_{4 B_0} \bigg( \frac{ \abs{u -
      \mean{u}_{4B_0}}^p}{R^p} \omega^p \bigg)^q \,dx \bigg)^{\frac
      1 {2q}}
    \\
    &\quad +
      c\, q^{p}
      \bigg( \int_{4 B_0} \big(
      \abs{\mathcal{V}(\cdot,G)}^2 \omega^p\big)^q\,dx \bigg)^{\frac 1
      {2q}}
    \\
    &\quad +
      c\, q\,\abs{B_0}^{\frac 1{2q}}
      \bigg( \dashint_{B_0}\abs{\mathcal{V}(\cdot,\nabla z)
      }^2\,dx
      \bigg)^{\frac 12}
  \end{align*}
  This implies
  \begin{align*}
    \bigg(\dashint_{B_0}\big( \abs{\nabla z}^p \omega^p \big)^q\,dx\bigg)^{\frac 1q}
    &\leq
      c\, q^{p}  \bigg( \dashint_{4 B_0} \bigg( \frac{ \abs{u -
      \mean{u}_{4B_0}}^p}{R^p} \omega^p \bigg)^q \,dx \bigg)^{\frac
      1 {q}}
    \\
    &\quad +
      c\, q^{p}
      \bigg( \dashint_{4 B_0} \big(
      \abs{G}^p \omega^p\big)^q\,dx \bigg)^{\frac 1
      {q}}
      +
      c\, q\,
      \dashint_{B_0}\abs{\mathcal{V}(\cdot,\nabla z)}^2\,dx.
  \end{align*}
  This proves the claim.
  Using the definition of~$z$ from~\eqref{eq:def-z} and
  \eqref{eq:def-g} we can translate this back to an estimate in terms
  of~$\nabla u$.
  \begin{align}
    \label{eq:main-pre-2}
    \begin{aligned}
      \bigg(\dashint_{\frac 12 B_0}\big( \abs{\nabla u}^p \omega^p
      \big)^q\,dx\bigg)^{\frac 1q} &\leq c\, q^{p} \bigg( \dashint_{4
        B_0} \bigg( \frac{ \abs{u - \mean{u}_{4B_0}}^p}{R^p} \omega^p
      \bigg)^q \,dx \bigg)^{\frac 1 q}
      \\
      &\quad + c\, q^{p} \bigg( \dashint_{4 B_0} \big( \abs{G}^p
      \omega^p\big)^q\,dx \bigg)^{\frac 1 q} + c\, q\,
      \dashint_{B_0}\abs{\nabla u}^p \omega^p\,dx
    \end{aligned}
  \end{align}
  Then the smallness assumption~\eqref{eq:smallness} on $\log \bbM$
  together with Proposition~\ref{pro:small-scalar} allows to apply the
  \Poincare{} type lemma of Proposition~\ref{pro:DreDur}. We obtain
  for some~$\theta \in (0,1)$
  \begin{align*}
    \begin{aligned}
      \bigg(\dashint_{\frac 12 B_0}\big( \abs{\nabla u}^p \omega^p
      \big)^q\,dx\bigg)^{\frac 1q} &\leq c_{q} \bigg( \dashint_{4 B_0}
      \big( \abs{\nabla u}^p \omega^p \big)^{\theta q} \,dx
      \bigg)^{\frac 1 {\theta q}} + c_{q} \bigg( \dashint_{4 B_0} \big(
      \abs{G}^p \omega^p\big)^q\,dx \bigg)^{\frac 1 q}.
    \end{aligned}
  \end{align*}
  We obtained a reverse H\"older's estimate for $(\abs{\nabla u}^p
  \omega^p)^q$. Now, Lemma~\ref{lem:imprevH} allows to reduce the
  exponent~$\theta q$ to $\frac 1p$. In particular, we get
  \begin{align}
    \label{eq:main-pre-3}
    \begin{aligned}
      \bigg(\dashint_{\frac 12 B_0} \big( \abs{\nabla u}^p \omega^p
      \big)^q\,dx\bigg)^{\frac 1q} &\leq c_{q} \bigg(\dashint_{4 B_0} 
      \abs{\nabla u} \omega \,dx\bigg)^p + c_{q} \bigg( \dashint_{4 B_0}
      \big( \abs{G}^p \omega^p\big)^q\,dx \bigg)^{\frac 1 q}.
    \end{aligned}
  \end{align}
  This proves the claim.
\end{proof}
We are now prepared to prove our main theorem,
\begin{proof}[Proof of main Theorem~\ref{thm:main}.]
  Propositon~\ref{pro:main-pre} agrees in most parts with our main
  theorem.  First, the Proposition~\ref{pro:main-pre} is stated with
  higher integrablity on $\frac 12 B_0$, right-hand side on~$4 B_0$
  and smallness on~$8 B_0$. A simple covering argument shows that we
  can replace this by higher integrbility on $B_0$, right-hand side on
  $2B_0$ and smallness on~$4B_0$.

  Second, we require in Proposition~\ref{pro:main-pre} the a~priori
  knowledge, that $\abs{\nabla u}^p \omega^p$ is already locally
  in~$L^q$. This artificial assumption can be overcome for example by
  an approximation argument. This way was for example used
  in~\cite{KinZho99} and~\cite{BCGOP17}, see also~\cite{CGP19}.   Due to our precise sharp
  estimates we are able to circumvent this argument and argue
  directly. Indeed, it follows by an iteration argument
  that~$\abs{\nabla u}^p \omega^p \in L^q(B_0)$. For this, let
  $q_1 \in [1, q_0]$ be such that
  $\abs{\nabla u}^p \omega^p \in L^{q_1}(B_0)$. Then
  Proposition~\ref{pro:main-pre} ensures that we have a reverse
  H\"older's estimate for $(\abs{\nabla u}^p \omega^p)^{q_1}$. The
  constants of this estimate only depend on~$q_0$ and are independent
  of~$q_1$. Therefore, we can apply Gehring's lemma
  (e.g. \cite[Theorem~6.6]{Giu03}) to deduce
  $\abs{\nabla u}^p \omega^p \in L^{s_1 q_1}$ with $s_1>1$ only
  depending on~$q$. Repeating this argument we see that
  $\abs{\nabla u}^p \omega^p \in L^{q_0}$ and
  Proposition~\ref{pro:main-pre} can be applied. Our main Theorem
  \ref{thm:main} follows.
\end{proof}

\subsection{Main Result Linear}

In this subsection we give the proof of the main
Theorem~\ref{thm:main-linear} for the linear setting.
\begin{proof}[Proof of main Theorem~\ref{thm:main-linear}.]
  The case $\rho \geq 2$ just follows from Theorem~\ref{thm:main} with
  $p=2$, so it remains to prove the case $1 < \rho < 2$. We will
  deduce this from the case $\rho > 2$ by means of a local duality
  argument.

  Recall that
  \begin{align*}
    -\divergence (\bbA(x) \nabla u)
    &=
      -\divergence ( \bbA(x) G)
  \end{align*}
  and that $B_0$ be a ball with radius~$R$ and $4B_0 \subseteq \Omega$.
  Let $H \in L^{\rho'}_\omega(2B_0)$ with
  \begin{align}
    \label{eq:H-constraint}
    \bigg( \dashint_{2B_0} (\abs{H}\,\omega)^{\rho'}\,dx \bigg)^{\frac
    1{\rho'}} &\leq 1.
  \end{align}
  Now, let $z$ solve the dual equation
  \begin{align}
    \label{eq:dualA}
    \begin{alignedat}{2}
      -\divergence(\bbA(x) \nabla z)&=-\divergence
      (\bbA(x)\indicator_{2B_0} H) &\quad & \text{on } 4{B_0},
      \\
      z&=0 &\quad& \text{on }\partial(4{B_0}).
    \end{alignedat}
  \end{align}
  We want to control~$\abs{\nabla z} \omega$ in terms of~$H$. For this
  we can  use the  super-quadratic case that we have already proven. In
  particular, by Theorem~\ref{thm:main-linear} applied to the
  exponent~$\rho' \geq 2$ we have
  \begin{align*}
    \bigg(\dashint_{2B_0} \big( \abs{\nabla z}\, \omega
    \big)^{\rho'}\,dx\bigg)^{\frac 1{\rho'}}
    &\leq c \dashint_{4 B_0} 
                                   \abs{\nabla z}\, \omega
                                         \,dx + c\, \bigg( \dashint_{2 B_0}
                                   \big( \abs{H}\, \omega\big)^{\rho'}\,dx
      \bigg)^{\frac 1 {\rho'}}
    \\
    &\leq c \bigg(\dashint_{4 B_0} 
                                   (\abs{\nabla z}\, \omega)^2
                                         \,dx \bigg)^{\frac 12} + c\, \bigg( \dashint_{2 B_0}
                                   \big( \abs{H}\, \omega\big)^{\rho'}\,dx
                                   \bigg)^{\frac 1 {\rho'}}. 
  \end{align*}
  Using the test function~$z$ in~\eqref{eq:dualA} we immediately see that
  \begin{align*}
    \bigg(\dashint_{4 B_0} 
    (\abs{\nabla z}\, \omega)^2
    \,dx \bigg)^{\frac 12} &\leq c\, \bigg( \dashint_{2 B_0}
                             \big( \abs{H}\, \omega\big)^2\,dx
                             \bigg)^{\frac 1 2} \leq c\, \bigg( \dashint_{2 B_0}
                             \big( \abs{H}\, \omega\big)^{\rho'}\,dx
                             \bigg)^{\frac 1 {\rho'}}.
  \end{align*}
  This and the previous estimate imply
  \begin{align}
    \label{eq:est-dual-z}
    \bigg(\dashint_{2B_0} \big( \abs{\nabla z}\, \omega
    \big)^{\rho'}\,dx\bigg)^{\frac 1{\rho'}}
    &\leq c\, \bigg( \dashint_{2 B_0}
                             \big( \abs{H}\, \omega\big)^{\rho'}\,dx
      \bigg)^{\frac 1 {\rho'}} \leq c.
  \end{align}
  We choose a cut-off function $\eta \in C^\infty_0(2B_0)$ with
  $\indicator_{B_0} \leq \eta \leq \indicator_{2B_0}$ and
  $\norm{\nabla \eta}_\infty \leq c R^{-1}$. Using the equation
  for~$z$ we calculate
  \begin{align*}
    \lefteqn{\textrm{I} := \dashint_{2B_0}
    \bbA(x) \nabla (\eta^2(u-u_0)) \cdot  H
    \, dx} \qquad
    \\
    &=
      \dashint_{2B_0}
      \bbA(x) \nabla (\eta^2(u-u_0)) \cdot  \nabla z
      \, dx
    \\
    % &=  
    %   \dashint_{2B_0} \eta^2 
    %   \bbA(x) \nabla u \cdot  \nabla z
    %   \, dx +
    %   \dashint_{2B_0}
    %   \bbA(x) \nabla (\eta^2) (u-u_0) \cdot  \nabla z
    %   \, dx 
    % \\
    &=        \dashint_{2B_0} 
      \bbA(x) \nabla u \cdot  \nabla (\eta^2 (z-z_0))
      \, dx +
      \dashint_{2B_0}
      \bbA(x) \nabla (\eta^2) (u-u_0) \cdot  \nabla z
      \, dx
    \\
    &\quad 
      -
      \dashint_{2B_0} 
      \bbA(x) \nabla u \cdot  \nabla (\eta^2) (z-z_0)
      \, dx
    \\
    &=: \textrm{I}_1 + \textrm{I}_2 + \textrm{I}_3.
  \end{align*}
  Using the equation for~$u$ we get 
  \begin{align*}
    \abs{I_1}
    &=  \biggabs{\dashint_{2B_0} \bbA(x) G \cdot \nabla (\eta^2
      (z-z_0))\,dx}
    \\           
    &\leq  \dashint_{2B_0} \omega^2 \abs{G} \bigabs{\nabla (\eta^2
      (z-z_0))}\,dx
    \\           
    & \leq \bigg(\dashint_{2B_0} (\omega \abs{G})^\rho \,dx \bigg)^{\frac{1}{\rho}}
      \bigg(\dashint_{2B_0} (\omega \abs{\nabla (\eta^2
      (z-z_0))})^{\rho'} \,dx \bigg)^{\frac{1}{\rho'}} .
  \end{align*}
  Using triangle inequality and  the weighted \Poincare{}'s inequality
  of Proposition~\ref{pro:DreDur}
  \begin{align*}
    \abs{\textrm{I}_1}
    & \leq \bigg(\dashint_{2B_0} (\omega \abs{G})^\rho \,dx \bigg)^{\frac{1}{\rho}}
      \bigg(\dashint_{2B_0} (\omega \abs{\nabla z})^{\rho'} \,dx \bigg)^{\frac{1}{\rho'}}.
  \end{align*}
Next
  \begin{align*}
    \abs{\textrm{I}_2}
    & \leq \bigg(\dashint_{2B_0} \bigg(\omega \frac{\abs{u-u_0}}{R} \bigg)^{ \rho}
      \,dx \bigg)^{\frac{1}{\rho}}
      \bigg(\dashint_{2B_0} (\omega \abs{\nabla z})^{\rho'} \,dx
      \bigg)^{\frac{1}{\rho'}} . 
  \end{align*}
  Thus, by  \Poincare{}'s inequality of Proposition~\ref{pro:DreDur}
  \begin{align*}
    \abs{\textrm{I}_2}
    & \leq \bigg(\dashint_{2B_0} \big(\omega \abs{\nabla
      u}\big)^{\theta \rho}
      \,dx \bigg)^{\frac{1}{\theta \rho}}
      \bigg(\dashint_{2B_0} (\omega \abs{\nabla z})^{\rho'} \,dx
      \bigg)^{\frac{1}{\rho'}}
  \end{align*}
  for some $\theta \in (\frac{1}{\rho},1)$. Moreover, for some
  $\theta_2 \in (0,1)$ close to one we have
  \begin{align*}
    \abs{\textrm{I}_3}
    & \lesssim
      \bigg(\dashint_{2B_0} (\omega \abs{\nabla u})^{\theta_2\rho} \,dx
      \bigg)^{\frac{1}{\theta_2 \rho}} \bigg(\dashint_{2B_0}
      \bigg(\omega \frac{\abs{z-z_0}}{R} \bigg)^{ (\theta_2 \rho)'}
      \,dx \bigg)^{\frac{1}{(\theta_2 \rho)'}}.
  \end{align*}
  Again, by \Poincare{}'s inequality of Proposition~\ref{pro:DreDur}
  with $\theta_2$ close to one we get
  \begin{align*}
    \abs{\textrm{I}_3}
    & \lesssim
      \bigg(\dashint_{2B_0} (\omega \abs{\nabla u})^{\theta_2\rho} \,dx
      \bigg)^{\frac{1}{\theta_2 \rho}} \bigg(\dashint_{2B_0}
      \big(\omega \abs{\nabla z} \big)^{ \rho'}
      \,dx \bigg)^{\frac{1}{\rho'}}.
  \end{align*}
  It is possible to choose~$\theta = \theta_2$ in the above steps. We
  finally obtained
  \begin{align*}
    \abs{\textrm{I}} &\lesssim \Bigg[ \bigg(\dashint_{2B_0} (\omega \abs{\nabla u})^{\theta\rho} \,dx
                       \bigg)^{\frac{1}{\theta \rho}} + \bigg(\dashint_{2B_0} (\omega
                       \abs{G})^\rho \,dx \bigg)^{\frac{1}{\rho}}
                       \Bigg] \bigg(\dashint_{2B_0}
                       \big(\omega \abs{\nabla z} \big)^{ \rho'}
                       \,dx \bigg)^{\frac{1}{\rho'}}.
  \end{align*}
  With~\eqref{eq:est-dual-z}  we get
  \begin{align*}
    \abs{\textrm{I}} &\lesssim \Bigg[ \bigg(\dashint_{2B_0} (\omega \abs{\nabla u})^{\theta\rho} \,dx
                       \bigg)^{\frac{1}{\theta \rho}} + \bigg(\dashint_{2B_0} (\omega
                       \abs{G})^\rho \,dx \bigg)^{\frac{1}{\rho}}
                       \Bigg] \bigg(\dashint_{2B_0}
                       \big(\omega \abs{\nabla z} \big)^{ \rho'}
                       \,dx \bigg)^{\frac{1}{\rho'}}
    \\
                     &\lesssim  \bigg(\dashint_{2B_0} (\omega \abs{\nabla u})^{\theta\rho} \,dx
                       \bigg)^{\frac{1}{\theta \rho}} + \bigg(\dashint_{2B_0} (\omega
                       \abs{G})^\rho \,dx \bigg)^{\frac{1}{\rho}}.
  \end{align*}
  Since $H$ was arbitrary satisfying~\eqref{eq:H-constraint} and
  $(L^{\rho'}_\omega)^* = L^\rho_{\omega^{-1}}$, it follows that
  \begin{align*}
    \bigg(\dashint_{2B_0} \abs{\bbA \nabla (\eta^2(u-u_0))} \omega^{-1}\big)^{\rho}\,dx
    \bigg)^{\frac 1 \rho}
    &\lesssim  \bigg(\dashint_{2B_0} (\omega \abs{\nabla u})^{\theta\rho} \,dx
      \bigg)^{\frac{1}{\theta \rho}} + \bigg(\dashint_{2B_0} (\omega
      \abs{G})^\rho \,dx \bigg)^{\frac{1}{\rho}}.
  \end{align*}
  Using $\bbA \nabla (\eta^2 (u-u_0)) = \bbA \nabla u$ on $B_0$ and 
  $\abs{\bbA \nabla u} \eqsim \omega^2 \abs{\nabla u}$, we obtain
  \begin{align*}
    \bigg(\dashint_{B_0} \big(\omega \abs{\nabla u}\big)^{\rho}\,dx
    \bigg)^{\frac 1 \rho}
    &\lesssim  \bigg(\dashint_{2B_0} (\omega \abs{\nabla u})^{\theta\rho} \,dx
      \bigg)^{\frac{1}{\theta \rho}} + \bigg(\dashint_{2B_0} (\omega
      \abs{G})^\rho \,dx \bigg)^{\frac{1}{\rho}}.
  \end{align*}
  Now, Lemma~\ref{lem:imprevH} allows to reduce the
  exponent~$\theta \rho$ to $1$. This proves
  Theorem~\ref{thm:main-linear} also in the sub-quadratic case~$p <2$.
\end{proof}

\section{\texorpdfstring{Sharpness of the $\log$-$\setBMO$ Condition}{Sharpness of the log-BMO Condition}}
\label{sec:counterexample}

In this section we show by means of examples that our $\log$-$\setBMO$
condition is sharp. In particular, we show in Example~\ref{exa:meyers}
that that the condition on the exponent~$\rho$ of higher integrability
$\abs{\log \bbM}_{\setBMO(8B_0)} \leq \frac{\kappa}{\rho}$ in
Theorem~\ref{thm:main-linear} and Theorem~\ref{thm:main} is optimal.
We also present an example with a degenerate a weight which does not
belong to~$\setBMO$, but satisfies our smallness
condition~$\abs{\log \bbM}_{\setBMO} <\epsilon$ (see Example
\ref{exp:absx-eps}).

Our examples are formulated for the nice linear situation, i.e. $p=2$,
which shows that Theorem~\ref{thm:main} is even optimal in the linear
case, which corresponds to Theorem~\ref{thm:main-linear}
for~$\rho > 2$.

Before we start with our examples let us make a short remarks on the
logarithm of certain matrices. If $a \in (-1,1)$ and $x \in \Rn$, then
by Taylor expansion
\begin{align*}
  \log(\identity + a \hatxx) &= \sum_{k \geq 1} \frac{(-1)^{k+1}}{k}
                               (a \hatxx)^k
  \\
                             &= \Big(\sum_{k \geq 1} \frac{(-1)^{k+1}}{k} a^k\Big) \hatxx =
                               \log(1+a) \hatxx.
\end{align*}
It is possible to conclude from this that for all $a>-1$
\begin{align}
  \label{eq:log-idxx}
  \log(\identity + a \hatxx) &= \Big(\sum_{k \geq 1}
                               \frac{(-1)^{k+1}}{k} a^k\Big) \hatxx = \log(1+a) \hatxx.
\end{align}
We will also need that the matrix $\identity + a\,\hatxx$ has
eigenvalues $1+a$ with eigenvector~$\hat{x} := x/\abs{x}$ and
eigenvalue $1$ with
eigenspace $(\linearspan \set{\hat{x}})^\perp$. This implies that for
example that $\idhatxx$ has eigenvalues zero and one, so for the
spectral norm we have~$\abs{\idhatxx}=1$.

\begin{example}
  \label{exa:meyers}
  This example is a modification of the one of
  Meyers~\cite[Section~5]{Mey63}, who considered the case~$n=2$.  Let
  $B_1(0)$ denote the unit ball in~$\setR^n$. Let us define
  $u\,:\, B_1(0) \to \setR$ for $n \geq 2$ by
  \begin{align}
    \label{eq:meyer-u}
    u(x) &:= \abs{x}^{1-\epsilon} \hat{x}_1
  \end{align}
  with $\hat{x}_1 := x_1/\abs{x}$ and $\epsilon \in (0,\frac 12]$.  Then
  \begin{align*}
    \nabla u(x)
    &= \abs{x}^{-\epsilon} \big(e_1 -\epsilon \hat{x} \hat{x}_1\big),
  \end{align*}
  where $e_1=(1,0,\dots, 0)$.  Since $\epsilon \in (0,\frac 12]$ we
  have $u \in W^{1,2}(B_1(0))$. More precisely, we have
  $\nabla u \in L^{\frac{n}{\epsilon},\infty}(B_1(0))$ (Marcinkiewicz space),
  so $u \in W^{1,\rho}(B_1(0))$ for all $\rho < \frac{n}{\epsilon}$. Moreover,
  $u \notin W^{1,\rho}(B_1(0))$ for $\rho \geq \frac{n}{\epsilon}$.

  Let us define the symmetric
  matrices~$\bbM, \bbA\,:\, \setR^n \to \setR^{n \times n}$ by
  \begin{align}
    \label{eq:meyer-MA}
    \begin{aligned}
      \bbM(x) &:= \theta \identity + (1-\theta) \hatxx,
      \\
      \bbA(x) &:= \bbM^2(x) = \theta^2 \identity + (1-\theta^2) \hatxx
    \end{aligned}
  \end{align}
  with $\theta \in (0,1)$ to be chosen later.  The eigenvalues of
  $\bbM$ are~$1$ with eigenvector~$\hat{x}$ and~$\theta$ with
  multiplicity~$n-1$ and eigenspace
  $(\linearspan \set{\hat{x}})^\perp$. For~$\bbA$ the
  eigenvalue~$\theta$ changes to~$\theta^2$.  Thus,
  \begin{align}
    \label{eq:meyer-elliptic}
    \begin{aligned}
      \theta \identity \leq \bbM(x) \leq \identity,
      \\
      \theta^2 \identity \leq \bbA(x) \leq \identity.
    \end{aligned}
  \end{align}
  We calculate
  \begin{align*}
    \bbM(x)^2 \nabla u(x) = \bbA(x) \nabla u(x)
    &= \abs{x}^{-\epsilon} \Big(\theta^2 e_1 + (1-\epsilon-\theta^2) \hat{x} \hat{x}_1
      \Big)
    \\
    \intertext{and}
    -\divergence(\bbM^2 \nabla u) = 
    -\divergence(\bbA \nabla u)
    &= -\abs{x}^{-\epsilon-1} \big(( -\epsilon(1-\epsilon) + (1-\epsilon- \theta^2)(n-1)\big)  \hat{x}_1.
  \end{align*}
  To get $-\divergence(\bbA \nabla u)=0$ we need
  \begin{align}
    \label{eq:mu-theta}
    -\epsilon(1-\epsilon) + (1-\epsilon- \theta^2)(n-1) &=0.
  \end{align}
  For $n=2$ we can set $\theta := 1-\epsilon$. In general, we can define
  \begin{align}
    \label{eq:def-theta}
    \theta& := \sqrt{1-\epsilon-\frac{\epsilon(1-\epsilon)}{n-1}}
  \end{align}
  and obtain
  \begin{align*}
    -\divergence(\bbM^2 \nabla u) = 
    -\divergence(\bbA \nabla u) &=0.
  \end{align*}
  Since $\epsilon \in (0,\frac 12]$ and $n \geq 2$, we have
  $\theta \in [\frac 12,1)$. This implies
  with~\eqref{eq:meyer-elliptic} that
  \begin{align*}
    \abs{\bbM(x)} \abs{\bbM^{-1}(x)} &\leq \frac 1 \theta \leq 2 =: \Lambda.
  \end{align*}
  In particular, the condition number of~$\bbM(x)$ is bounded
  independently of the specific choice of~$\epsilon \in (\frac 12, 1]$.

  By~\eqref{eq:log-idxx} we calculate 
  \begin{align*}
    \log \bbM &= \log( \theta \identity + (1-\theta) \hatxx) =
                \log(\theta) \identity + \log\bigg(\identity +
                \frac{1-\theta}{\theta} \hatxx \bigg)
    \\
              &=
                \log(\theta) \identity + \log\bigg( 1+
                \frac{1-\theta}{\theta}\bigg) \hatxx
    \\
              &= \log(\theta) (\idhatxx).
  \end{align*}
  Thus,
  \begin{align*}
    \norm{\log \bbM}_{\setBMO} &\leq  \norm{\log \bbM}_\infty =
                                 \abs{\log(\theta)} 
                                 \abs{\idhatxx} = \abs{\log(\theta)}. 
  \end{align*}
  Thus, by~\eqref{eq:def-theta}
  \begin{align*}
    \abs{\log \theta}
    &= \Biggabs{\log \sqrt{1-\epsilon-\frac{\epsilon(1-\epsilon)}{n-1}
      }}
      \leq \tfrac 12 \abs{\log (1- 2 \epsilon)} \leq \epsilon.
  \end{align*}
  Overall, we have a function $u\,:\, B \to \setR$ and a positive
  matrix valued weights~$\bbA=\bbM^2$ with the following properties
  \begin{enumerate}
  \item The function $u \in W^{1,2}(B)$ solves
    \begin{align*}
      -\divergence(\bbM^2 \nabla u) = 
      -\divergence(\bbA \nabla u) &=0
    \end{align*}
  \item The condition number of the weight satisfies
    $\abs{\bbM(x)} \abs{\bbM^{-1}(x)} \leq 2$.
  \item The weight satisfies the smallness condition
    $\norm{\log \bbM}_\infty \leq \epsilon$.
  \item We have limited higher integrability of the gradients. More
    precisely, we have $u \in W^{1,\rho}(B)$ for all $\rho\,\epsilon < n$
    and $u \not\in W^{1,\rho}(B)$ for all $\rho\, \epsilon \geq n$.
  \end{enumerate}
  This shows that the smallness assumption
  $\abs{\log \bbM}_{\setBMO} \leq \kappa_0 \frac{1}{\rho}$ in our
  Theorems~\ref{thm:main-linear} and~\ref{thm:main} are optimal.
\end{example}
We will now present an example with a degenerate matrix-valued
weight~$\bbA$, which is not from~$\setBMO$ but satisfies our
logarithmic smallness assumption. Note, that it was already mentioned
in~\cite[Remark~2.12]{CaoMenPha18} that the condition
$\bbA\in \setBMO$ is not necessary.
\begin{example}
  \label{exp:absx-eps}
  We proceed similar to Example~\ref{exa:meyers}.  Let $B_1(0)$ denote
  the unit ball in~$\Rn$ with $n\geq 2$. For
  $\epsilon \in (0, \frac 12]$ define $u\,:\, B_1(0) \to \setR$ by
  \begin{align*}
    u(x) &:= \abs{x}^{1-\epsilon/2} \hat{x}_1.
  \end{align*}
  with $\hat{x}_1 := x_1/\abs{x}$. Moreover, let us define the
  matrix-valued
  weights~$\bbM, \bbA\,:\, \setR^n \to \setR^{n \times n}$ by
  \begin{align*}
    \bbM(x) &:= \abs{x}^{-\epsilon/2}(\theta \identity + (1-\theta) \hatxx),
    \\
    \bbA(x) &:= \bbM^2(x) =\abs{x}^{-\epsilon/2} (\theta^2 \identity +
              (1-\theta^2) \hatxx). 
  \end{align*}

  So compared to Example~\ref{exa:meyers} our function~$u$ has an
  additional factor~$\abs{x}^{\epsilon/2}$ and the weight~$\bbM$ has an
  additional factor~$\abs{x}^{-\epsilon/2}$. Similar calculations lead to
  \begin{align}
    \label{eq:exa2-u}
    \nabla u(x)
    &= \abs{x}^{-\epsilon/2} \Big(e_1 -\frac \epsilon2 \hat{x}
      \hat{x}_1\Big)
  \end{align}
  with $e_1=(1,0,\dots, 0)$ and
  \begin{align*}
    -\divergence(\bbM^2 \nabla u) = 
    -\divergence(\bbA \nabla u)
    &= -\abs{x}^{-\epsilon-1} \Big( -\frac \epsilon 2 (1-\epsilon) + \big(1- \frac
      \epsilon2- \theta^2 \big)(n-1)\Big)  \hat{x}_1.
  \end{align*}
  Thus, for 
  \begin{align}
    \label{eq:def-theta1}
    \theta& := \sqrt{1-\frac \epsilon 2 - \frac{\epsilon(1-\epsilon)}{2(n-1)}}
  \end{align}
  we obtain
  \begin{align*}
    -\divergence(\bbM^2 \nabla u) = 
    -\divergence(\bbA \nabla u) &= 0.
  \end{align*}
  Since $\delta>0$ and $\epsilon \in (0, \frac 12]$ we have $\theta \in [\frac 12,
  1)$

  Our weight satisfies
  \begin{align*}
    \abs{x}^{-\epsilon/2} \theta \identity \leq \bbM(x) \leq 
    \abs{x}^{-\epsilon/2} \identity.
  \end{align*}
  So, although the weight~$\bbA$ is singular, it has finite condition
  number
  \begin{align*}
    \abs{\bbM(x)} \abs{\bbM^{-1}(x)} &\leq \frac 1 \theta \leq 2 =:\Lambda.
  \end{align*}
  Similar to Example~\ref{exa:meyers} we conclude
  \begin{align*}
    \log \bbM
    &=- \frac \epsilon 2 (\log\abs{x}) \identity +\log( \theta) (\idhatxx).
  \end{align*}
  Thus,
  \begin{align*}
    \abs{\log \bbM}_{\setBMO} &\leq \frac \epsilon 2 \bigabs{\log \abs{x}
                           \identity}_{\setBMO} + \abs{\log(\theta)
                           (\idhatxx)}_{\setBMO}
                           \leq \epsilon + \abs{\log(\theta)}.
  \end{align*}
  We calculate
  \begin{align*}
    \abs{\log(\theta)}
    &=
      \tfrac 12 \biggabs{\log
      \bigg(
      1-\frac \epsilon 2 - \frac{\epsilon(1-\epsilon)}{2(n-1)}\bigg)}
      \leq \tfrac 12  \bigabs{\log(1-\epsilon)} \leq \frac \epsilon 2.
  \end{align*}
  Overall,
  \begin{align*}
    \abs{\log\bbM}_{\setBMO} &\leq \tfrac 32 \epsilon.
  \end{align*}
  This shows that the weight~$\bbM$ satisfies our $\log$-smallness
  condition and our Theorem~\ref{thm:main} can be applied. In
  particular, we get $\nabla u \in L^\rho(B)$ for all $q>p$ with
  $\rho \leq \frac{\kappa_0}{\epsilon}$.

  Due to~\eqref{eq:exa2-u} we have 
  $\nabla u \in L^{\frac{2n}{\epsilon},\infty}(B)$ (Marcinkiewicz space). More
  precisely,  $\nabla u \in L^2$ if and only if $\rho < \frac{2n}{\epsilon}$. So
  we have limited higher integrability in agreement with our
  Theorems~\ref{thm:main-linear} and~\ref{thm:main}.
\end{example}

\end{document}